\documentclass[final]{templates/SciChinaMath/SCIYA2023enOL}
\usepackage{graphicx}
\usepackage{amssymb}
\usepackage{amsthm}
\usepackage{amsmath}
\usepackage{bm}
\usepackage{enumerate}
\usepackage{nicematrix}
\usepackage{mathrsfs} % mathscript
\usepackage{input/boondox-cal} % lower mathcal
\usepackage{booktabs} % three-line table
\usepackage{multirow}
\usepackage{url}
\usepackage{lineno}
\usepackage{hyperref}
\usepackage{diagbox}
\hypersetup
{
    colorlinks=true,
    linkcolor=blue,
    filecolor=blue,
    urlcolor=blue,
    citecolor=cyan,
}
\usepackage{verbatim}
\newcommand\norm[2]{\left\Vert #1 \right\Vert _{#2}} % Norm
\newcommand\innprod[2]{\left\langle #1,#2 \right\rangle} % Inner product
\newcommand\floor[1]{\lfloor #1 \rfloor} % Floor function
\newcommand\ceil[1]{\lceil #1 \rceil} % Ceiling function
\newcommand\ints[0]{\cap} % Intersection
\newcommand\fform[0]{\mathlcal{f}} % f-form operator
\newcommand\R[0]{\mathbb{R}}

\newcommand\N[0]{\mathbb{N}}

\theoremstyle{plain}
  \newtheorem{thm}{Theorem}[section]
  \newtheorem{lem}[thm]{Lemma}
  \newtheorem{prop}[thm]{Proposition}
  
\theoremstyle{definition}
  \newtheorem{defn}[thm]{Definition}

  \newtheorem{rem}[thm]{Remark}

\begin{document}
\ensubject{mathematics}

%%%%%%%%%%%%%%%%%%%%%%%%%%%%%%%%%%%%%%%%%%%%%%%%%%%%%%%
%%% Authors do not modify the information below
\ArticleType{ARTICLES}
%\SpecialTopic{Progress of Projects Supported by NSFC}
%\SubTitle{Dedicated to Professor Yang Lo on the Occasion of his {\rm 70}th Birthday}
\Year{2023}
\Month{October}%
\Vol{65}
\No{1}
\BeginPage{1} %
\DOI{Manuscript for review}
%\ReceiveDate{January 1, 2022}
%\AcceptDate{January 1, 2022}
%\OnlineDate{January 1, 2022}
%%%%%%%%%%%%%%%%%%%%%%%%%%%%%%%%%%%%%%%%%%%%%%%%%%%%%%%

%%% title:
%%%   \title{title}{title for citation}
\title{Native spaces and generalization of Wu functions}
{Native spaces and generalization of Wu functions}

%%% Corresponding author:
%%%   \author[number]{Full name}{{email@xxx.com}}
%%% General author:
%%%   \author[number]{Full name}{}
\author[1,2]{Yixuan Huang}{yxh@mail.bnu.edu.cn}
\author[3,4]{Zongmin Wu}{zmwu@fudan.edu.cn}
\author[5,6]{Shengxin Zhu*}{{Shengxin.Zhu@bnu.edu.cn}}

%%% Author information for page head.
\AuthorMark{Huang Y.}

%%% Authors for citation.
\AuthorCitation{Y. Huang, Z. Wu, S. Zhu, et al}

%%% Address.
%%%   \address[number]{Address, City {\rm Postcode}, Country}
\address[1]{School of Mathematics, Beijing Normal University, China}
\address[2]{Laboratory of Mathematics and Complex Systems, Ministry of Education, China}
\address[3]{Shanghai Key Laboratory for Contemporary Applied Mathematics,\\ School of Mathematical Sciences, Fudan University, Shanghai, China;\\ Shanghai Center for Mathematical Sciences, Shanghai, China}
\address[4]{School of Big Data and Statistics, Anhui University, Hefei, China}
\address[5]{Research Center for Mathematics, Advanced Institute of Natural Science, \\ Beijing Normal University, Zhuhai {\rm 519087}, China}
\address[6]{Guangdong Provincial Key Laboratory of Interdisciplinary Research and Application for Data Science,\\ BNU-HKBU United International College, Zhuhai {\rm 519087}, China}

%\linenumbers
% Main body
\abstract{We prove that the native space of a Wu function is a dense subspace of a Sobolev space. An explicit characterization of the native spaces of Wu functions is given. Three definitions of Wu functions are introduced and proven to be equivalent. Based on these new equivalent definitions and the so called $f$-form tricks, we can generalize the Wu functions into the even-dimensional spaces $\R^{2k}$, while the original Wu functions are only defined in the odd-dimensional spaces $\R^{2k+1}$. Such functions in even-dimensional spaces are referred to as the `missing Wu functions'. Furthermore we can generalize the Wu functions into `fractional'-dimensional spaces. We call all these Wu functions the generalized Wu functions. The closed form of the generalized Wu functions are given in terms of hypergeometric functions. Finally we prove that the Wu functions and the missing Wu functions can be written as linear combinations of the generalized Wendland functions.
}

%%% Keywords.


\maketitle

%%%%%%%%%%%%%%%%%%%%%%%%%%%%%%%%%%%%%%%%%%%%%%%%%%%%%%%
%%% The main text.
%%%%%%%%%%%%%%%%%%%%%%%%%%%%%%%%%%%%%%%%%%%%%%%%%%%%%%%
\section{Introduction}
    The late 1980s and early 1990s have witnessed a thrive of \emph{radial basis function (RBF)} network analysis \cite{PJ93AP,HC93AP,MH92AP}. The \emph{radial symmetry} is a welcomed property to handle high dimensional problems. It has been noticed in the 1990s that RBF interpolation with lower order polynomials (or learning algorithms with regularization \cite{GF90NE}) are equivalent to multi-layer networks \cite{PT90RE}. The seminal rigorous error analysis of such an approximation has been developed by Wu and Schaback \cite{WZ93LO}. To \emph{sparsify} or \emph{prune} a dense network for better efficiency, \emph{compactly supported} functions are preferred. Wu firstly constructed a class of \emph{positive definite} compactly supported radial basis functions (CSRBFs) in terms of polynomials \cite{WZ95CO}. Positive definite radial basis functions result in symmetric positive definite interpolation matrices which are easy to solve. A radial basis function is positive definite if its Fourier transform is positive. Using Fourier transform to characterize a positive definite function dates back to Mathias \cite{MM23UB}, Bochner \cite{BS33MO} and followed by von Neumann and Schoenberg \cite{NJ41FO}, and among others. It provides a simple way to study the property of the interpolation system. One of the most famous positive definite radial basis functions is the Gaussian $\exp(-\varepsilon^2\norm{x}{2}^2)$ with a positive Fourier transform $(\sqrt{2} \varepsilon)^{-d}\exp(-\norm{\omega}{2}^2/(4\varepsilon^2))$. A Gaussian is a positive definite function in all integer-dimensional spaces. However, it had been tremendously difficult to find a compactly supported positive definite function in terms of polynomials until \cite{WZ95CO}. It had kept many leading researchers busy for many years. The difficulty of finding such functions may lie in two questions. On one hand, in general, it is not easy to find the Fourier transform of a compactly supported function in high dimensional spaces. On the other hand, positive definiteness depends on dimension. A positive definite function in $\R^k$ is not necessarily positive definite in $\R^m$ for $m\neq k$. For example,  the \emph{truncated power functions}, which are also called \emph{Askey's power functions} \cite{AR73RA,LG86AF}
    \begin{linenomath}\begin{equation}
        \varphi_\ell(x)=(1-\norm{x}{2})_+^\ell= 
        \begin{cases}
        (1-\norm{x}{2})^\ell &\mbox{ for }1-\norm{x}{2}\geq 0,\\
        0 &\mbox{ for }1-\norm{x}{2}\leq 0.\end{cases}
        \label{def:askey}
    \end{equation}\end{linenomath}
    Such a function can also be represented as the $ReLU((1-\norm{x}{2})^\ell)$ \emph{activation function} in machine learning. It is positive definite on $\R^d$ if $\ell\geq\floor{d/2}+1$, where $\floor{\cdot}$ is the floor function. 

    In Wu's neat construction, he introduced two operators which significantly remove many involved computations with Fourier transforms in high dimensional spaces. One is the integral operator 
    \begin{linenomath}\begin{equation}
        \mathscr{I}\varphi(r):=\int_r^\infty t\varphi(t)\ dt,\mbox{ for }r\geq 0 \mbox{ and }t\varphi(t)\in L^1(\R_+),\label{eq:dim_walk_I}
        \end{equation}\end{linenomath}
    and the other is the differential operator
    \begin{linenomath}\begin{equation}
        \mathscr{D}\varphi(r):=-\frac{1}{r}\frac{d}{dr}\varphi(r),\mbox{ for }r\geq 0\mbox{ and }\varphi\in C^2(\R_+).\label{eq:dim_walk_D}
        \end{equation}\end{linenomath}
    These two operators enjoy the so called \emph{dimensional walk} properties of Fourier transform (see bellow Theorem \ref{thm:walk}). Precisely, let $\Phi(x)=\varphi(\norm{x}{2})\in L^1(\R^d)$ be a radial function and $\mathscr{F}_d\varphi(\norm{\omega}{2}):=\hat{\Phi}(\omega)$ be the $d$-variate Fourier transform of $\Phi$. Then under some mild conditions, $\mathscr{F}_d(\varphi) =\mathscr{F}_{d-2}(\mathscr{I}\varphi)$ and $ \mathscr{F}_d(\varphi)=\mathscr{F}_{d+2}(\mathscr{D}\varphi)$. This implies that if we have a positive definite function in $\R$, then we can construct positive definite functions with the $\mathscr{D}$ or $\mathscr{I}$ operator in $\R^{2k+1}$. 
  
    Recall that the Fourier transform of the convolution of two functions is the product of their Fourier transforms. Then it is easy to construct a positive definite function in $\R$ by self-convolution. Wu's construction started with a very smooth function by self-convolution $g_\ell:=f_\ell*f_\ell$, where $f_\ell=(1-r^2)_+^\ell$ is a truncated function in $\R$. Then according to the dimension walk property, $\mathscr{D}g_\ell$ is a positive definite CSRBF in $\R^3$ and $\mathscr{D}^2g_\ell$ is a positive definite CSRBF in $\R^5$, if $\ell$ is big enough. In such a way, Wu constructed a class of positive definite CSRBFs in terms of polynomials in $\mathbb{R}^{2k+1}$. Notice that he started with a very smooth function $g_\ell$ by choosing a big enough $\ell$, and then obtained some polynomials of lower degrees that are positive definite in higher dimensions. In this way, he constructed a piecewise polynomial of degree $4\ell-2k+1$ that is positive definite in $\R^{2k+1}$ and in $C^{2\ell-2k}$. Therefore, he proposed the interesting question on how to construct positive definite CSRBFs with minimal polynomial degree when prescribed space dimension and smoothness. Wendland solved this problem by starting with the Askey's power functions and applying the integral operator $\mathscr{I}$. He chose $\ell$ such that the Askey's power function is positive definite on $\R^d$. After applying the operator $\mathscr{I}$ on Askey's power function, he obtained a smoother positive definite CSRBF on a lower dimension. In such a way the $\mathscr{I}$ operator brings the later well-known \emph{Wendland functions} \cite{WH95PI}, which are compactly supported piecewise polynomials of \emph{minimal degrees} when given space dimension and smoothness. More precisely, when prescribed space dimension $2k+1$ and smoothness $C^{2\ell-2k}$, the corresponding Wendland function is of degree $3\ell-2k+1$, which is $\ell$ degrees lower than Wu's construction.
  
    Soon Wu found that interpolation with positive definite CSRBFs of the same scale violates the so-called \emph{Strang-Fix condition} \cite{SG71AF,WZ97CO}, which is a desirable condition for a good approximation. Another counterexample is to approximate constant functions with Gaussian summations \cite{ZS20SU}. This is one of the essential reasons why learning with CSRBFs should go to deep: multi-scale \cite{WH17MU,GQ10MU}, multi-level \cite{NF99MU}, multi-steps \cite{FM96MU}, or go to variable shapes \cite{ZS15CO,BM15IN} or \emph{numerical basis} functions with \emph{homogenization} \cite{LX21GE}. In addition to the applications in approximation and learning theory, positive definite RBFs also provide tools for probability theory as characteristic functions of spherically symmetric probability measures \cite{GT01CR}. For a review on the history of compactly supported RBFs, the reader is referred to a noticeable technique report \cite{ZS12CO}\cite[Chapter 2]{ZS14NU} and publication \cite{GT02CO}. Now RBFs have been proved to be a powerful tool in the field of high dimensional scattered data approximation \cite{WH04SC}. Examples of successful application include digital terrain modelling \cite{HR71MU}, mesh-free methods for solving partial differential equations \cite{FE13AH,FG07ME,FZ14RE},  machine learning \cite{SR06KE,GM01CL,ZD03CA,XY19GE}, neural networks \cite{BA96ME,FJ96RA,MG18RA,PT90RE}, etc.. In particular, the \emph{Hermite-Birkhoff interpolation} serves the foundation of the \emph{symmetric collocation} methods \cite{WZ92HE}. And the \emph{moving least square theory} for RBF which guarantees the Strang-Fix condition is closely related to the increasingly important \emph{smoothed particle hydrodynamics} methods \cite{SS14FE,SS16EF}.

    Among these theoretical and applicational advances related to RBFs, Wu's construction plays a unique role in RBF theory, not only because that it is the first construction of positive definite CSRBFs in terms of polynomials, but also because that the tools introduced in the construction have resulted in many other influential works. However, 30 years after Wu's construction, the mathematical theory on Wu functions is neither crystal clear nor complete:
    \begin{enumerate}[(1)]
        \item There is a shortage of explicit characterization of the \emph{native spaces} of Wu functions. Or can a Wu function recover a Sobolev space? 
        \item The original Wu functions were only given in odd-dimensional spaces because the step size of dimension walk is $2$, and functions in even-dimensional spaces are \emph{missing}. How can one generalize the Wu functions to complete the mathematics?
        \item Is there any closed formulas for Wu functions and what is the relationship between Wu functions and Wendland functions?
    \end{enumerate}
    Similar problems related to Wendland functions have been solved by Buhmann \cite{BM00AN}, Schaback \cite{SR11MI} and Hubbert \cite{HS12CL}. Whereas problems related to Wu functions still remain. One reason may lie in that the seminal work by Wu and Schaback \cite{SR96OP,WZ95CO} involves sophisticated mathematics, despite it is well cited, less people can really follow its details.  Besides, Gneiting pointed out that Wu's construction allows a larger space of scale mixtures \cite{GT01CR}, thus the local approximation spaces with Wu functions involves more general positive definite radial functions. From this perspective, it is also necessary to understand the Wu functions better.
    
    Following Wu's \cite{WZ95CO} and Schaback's \cite{SR11MI} work, this paper is devoted to fill ALL these gaps and answer ALL these questions. It turns out that the \emph{native space} of a Wu function is a \emph{dense subspace} of a Sobolev space. The inclusion (Theorem \ref{thm:wu_subspace_sobolev}) is obtained by comparing the kernel functions of the two spaces in the \emph{phase domain}. The Sobolev space is defined (in \eqref{eq:Hsd}) according to the \emph{harmonic analysis} approach, by the Fourier transform of the kernel functions. This seems differ with but is equivalent to the commonly used definition by weak derivatives in the numerical analysis community. And the difficulty lies in that how to find the Fourier transform of a Wu function (Theorem \ref{thm:fourier_wu}) and how to compare it with that of a Sobolev kernel (in \eqref{eq:imq}). The Fourier transform of Wu functions were in fact already given by Wu and Schaback in \cite[Sec. 7]{SR96OP}. Here we simplify computations via two tricks: we first slightly modify Wu's construction by re-scaling for simplicity; and then we employ the so called \emph{$f$-form operator} tricks introduced by Schoenberg \cite{SI38ME} in an early paper in \emph{Ann. Math.}. The denseness (Theorem \ref{thm:dense}) is based on an explicit characterization of the native spaces of Wu functions (Theorem \ref{thm:native_space_wu}). The explicit characterization is based on a characterization of the native spaces with kernels constructed by self-convolution (Theorem \ref{thm:native_space_conv}). All these derivations are related to some classical results on the \emph{reproducing kernel Hilbert space} \cite{AN50TH} and native space \cite{SR99NA,SR00AU,WH04SC}. 

    The $f$-form tricks can also transform the integral operator $\mathscr{I}$ and the differential operator $\mathscr{D}$ to two more general operators (Proposition \ref{prop:I1_I-1}). In particular, we can construct the $\mathscr{D}^{1/2}$ and $\mathscr{I}^{1/2}$ operators via the two more general operators and the $f$-form tricks, they can `walk' through all integer-dimensional spaces (Theorem \ref{thm:gen_walk}) and lead to the `missing' Wu functions. This makes the generalization of Wu functions possible. The generalization remains the key properties of the original Wu functions (Theorem \ref{thm:gen_wu_func}). The closed form of generalized Wu functions are given in terms of hypergeometric series (Theorem \ref{thm:closed_form_wu_func}). And the derivation of this closed form is based on the explicit formulas of special Wu functions of the form $\varphi_{\ell,\ell}$ (Theorem \ref{thm:phill}). Using the closed forms, we can give an explicit representation of the generalized Wendland functions in terms of linear combinations of generalized Wu functions (Theorem \ref{thm:wu_wendland}).
    
   % Applying such a trick on the operators $\mathscr{I}$ and $\mathscr{D}$, one can obtain the standard integration and differential operators which can walk through all the dimensional spaces and even fractional spaces (Theorem \ref{}). Together with some consistence conditions in Theorem \ref{}(4.1), it makes the `missing Wu functions' and generalized Wu functions possible. It turns out that closed form of these generalized Wu functions can be written in term of hypergemetric functions.  Some of the classic sophisticated special functions are involved in the whole analysis process. 

  %  Besides, some sophisticated mathematics for finding so-called \emph{$f$-form operator} and its inverse which was first introduced by Schoenberg \cite{SI38ME} in the 1930s. And the Fourier transform of Wu functions plays a key role in characterising its native space. We start with comparing the native spaces of Wu functions with Sobolev spaces, and then derive a convolutional representation for the native spaces of Wu functions. On the way to this end, we will meet with three equivalent definitions of Wu functions, which will eventually lead to the generalization of Wu functions.
    
% section 1

\section{Radial functions and transforms}
    
    Before diving deep into the question, we introduce some notations and results in \cite{SR11MI,SR96OP}.  
    Given a function $\varphi$ defined on $\R_+=[0,\infty)$ and a positive integer $d$, a $d$-variate function 
    \begin{linenomath}\begin{equation*}
        \Phi(x)=\varphi(\norm{x}{2}),\quad x\in\R^d,
        \end{equation*}\end{linenomath}
    arises as a radial function in $\R^d$. The Fourier transform of $\Phi$, denoted by $\hat{\Phi}$, is
    \begin{linenomath}\begin{equation*}
        \hat{\Phi}(\omega)=(2\pi)^{-d/2}\int_{\R^d}\Phi(x)\exp(-i\innprod{\omega}{x})\ dx,\quad\omega\in\R^d.
        \end{equation*}\end{linenomath}
    Computing the Fourier transform of an arbitrary function is not an easy task. But for radial function $\Phi(x)=\varphi(\norm{x}{2})$, things are much simpler because its Fourier transform $\hat{\Phi}$ is still a radial function \cite[Cor. 1.2]{SE75IN}, and $\hat{\Phi}$ could be computed with a one-dimensional integral involving $\varphi$ \cite[Thm 3.3]{SE75IN}
    \begin{linenomath}\begin{equation*}
        \begin{gathered}
            \hat{\Phi}(\omega)=\mathscr{F}_d\varphi(\norm{\omega}{2}),\quad \omega\in\R^d,\\
            \mathscr{F}_d\varphi(r):=r^{-(d-2)/2}\int_0^\infty\varphi(t)t^{d/2}J_{(d-2)/2}(tr)\ dt,\quad r\in\R_+,\ 
            \end{gathered}
        \end{equation*}\end{linenomath}
    where $J_\nu$ is the \emph{Bessel function of the first kind}, $H_\nu$ is defined by 
    \begin{equation}\label{defn:Jnu_Hnu}
        (x/2)^{-\nu}J_\nu(x)=H_\nu(x^2/4)=\sum_{k=0}^\infty\frac{(-x^2/4)^k}{k!\Gamma(k+\nu+1)}=\frac{{_0F_1}(\nu+1;-x^2/4)}{\Gamma(\nu+1)},
        \end{equation}
    and ${_0F_1}$ is the \emph{generalized hypergeometric function} \cite[2.1]{BW35GE}. The function $\mathscr{F}_d\varphi$ defined on $\R_+$ is called the $d$-variate Fourier transform of $\varphi$. It turns out that the inverse Fourier transform of $\Phi$, denoted by $\Phi^\vee$, is identical with $\hat{\Phi}$ if $\Phi$ is a radial function \cite{SR96OP}, so there is no need for extra notation for the $d$-variate inverse Fourier transform.
    
    The convolution of two $d$-variate radial functions $\Phi(x)=\varphi(\norm{x}{2})$ and $\Psi(x)=\psi(\norm{x}{2})$ turns out to be a radial function, so we define the $d$-variate convolution operator by
    \begin{linenomath}\begin{equation*}
        \Phi*\Psi(x)=:\mathscr{C}_d(\varphi,\psi)(\norm{x}{2}),\quad x\in\R^d.
        \end{equation*}\end{linenomath}
    As a special case, consider the convolution of functions from the \emph{tempered radial test function space} $\mathscr{S}$ \cite{SR96OP}. This space consists of infinitely differentiable functions on $\R_+$ such that all derivatives decay faster than any polynomial at infinity, namely
    \begin{linenomath}\begin{equation*}
        \mathscr{S}=\left\{f\in C^\infty(\R_+):\forall \alpha,\beta\in\mathbb{N},\sup_{x\in\R_+}\abs{x^\alpha\cdot\frac{d^\beta f}{dx^\beta}}<\infty\right\}.
        \end{equation*}\end{linenomath}
    Notice that $\mathscr{S}$ is slightly different from the classical \emph{Schwartz space} since the functions from these two spaces are defined on different domains. For functions $\varphi$ and $\psi$ in $\mathscr{S}$, we can equivalently define the $d$-variate convolution operator by
    \begin{linenomath}\begin{equation*}
        \mathscr{C}_d(\varphi,\psi)(r):=(2\pi)^{d/2}\mathscr{F}_d[(\mathscr{F}_d\varphi)\cdot(\mathscr{F}_d\psi)](r),\quad r\in\R_+,
        \end{equation*}\end{linenomath}
    according to the convolution theorem \cite[Thm 5.16(2)]{WH04SC}.
    
    The operators $\mathscr{I}$ and $\mathscr{D}$, defined in \eqref{eq:dim_walk_I} and \eqref{eq:dim_walk_D} respectively, play essential roles in constructing CSRBFs in high dimensions. The \emph{dimension walk} property \cite{WZ95CO,WH04SC,SR96OP} associated with them introduces ways to find positive definite functions in higher and lower dimensions.
    \begin{thm}(\cite[Theorem 9.6]{WH04SC}).  \label{thm:walk}
        Suppose $\varphi\in C(\R_+)$.
        \begin{enumerate}[(1)]
            \item If $t\mapsto t^{d-1}\varphi(t)\in L^1(\R_+)$ and $d\geq 3$, then $\mathscr{F}_d\varphi=\mathscr{F}_{d-2}\mathscr{I}\varphi$.
            \item If $\varphi\in C^2(\R_+)$ and $t\mapsto t^d\varphi'(t)\in L^1(\R_+)$, then $\mathscr{F}_d\varphi=\mathscr{F}_{d+2}\mathscr{D}\varphi$.
        \end{enumerate}
    \end{thm}

    To reveal connections among the operators defined above, we introduce the $f$-form of a function and the $f$-form of an operator. The concept of $f$-form stems from Schoenberg's work in order to characterize completely monotone functions \cite{SI38ME}. Later it is developed by Wu and Schaback as an important strategy to construct positive definite RBFs in higher space dimensions \cite{SR96OP,SR11MI}, and the construction strategy is summarized by Hubbert \cite{HS12CL}. In this paper, we formulate the $f$-form of a function and an operator as follows.
    \begin{defn}\label{defn:f-form}
        Let $X:=\{f:\R_+\to\R\}$ be the set of all functions defined on $\R_+$. For any function $\varphi\in X$, the \emph{$f$-form} of $\varphi$, denoted by $\fform\varphi$, is defined by
        \begin{linenomath}\begin{equation*}
            \fform\varphi(r):=\varphi(\sqrt{2r}).
            \end{equation*}\end{linenomath}
        The \emph{inverse $f$-form} of $\varphi$, denoted by $\fform^{-1}\varphi$, is defined by
        \begin{linenomath}\begin{equation*}
            \fform^{-1}\varphi(r)=\varphi(r^2/2).
            \end{equation*}\end{linenomath}
        The operators $\fform$ and $\fform^{-1}$ are called the \emph{$f$-form operator} and \emph{inverse $f$-form operator} respectively. For any operator $A:X\to X$, the $f$-form of $A$ is defined by $\fform A\fform^{-1}$, that is, for any $\varphi\in X$,
        \begin{linenomath}
            \begin{equation*}
                (\fform A\fform^{-1})\varphi=\fform(A(\fform^{-1}\varphi)).
            \end{equation*}
        \end{linenomath}
        \end{defn}
    
    One reason for introducing the $f$-form operator is that it turns the operator $\mathscr{D}$ into the usual differential operator $-d/dr$, providing convenience for computation.
    \begin{prop}\label{prop:I1_I-1}
        Suppose $\varphi,\psi$ are functions on $\R_+$ such that $\fform\mathscr{D}\fform^{-1}\varphi$ and $\fform\mathscr{I}\fform^{-1}\psi$ exists. Then 
        \begin{linenomath}\begin{equation*}
                \begin{gathered}
                    \fform\mathscr{D}\fform^{-1}\varphi(r)=-\frac{d}{dr}\varphi(r)=:D\varphi(r)=I_{-1}\varphi(r),\\
                    \fform\mathscr{I}\fform^{-1}\psi(r)=\int_r^\infty \psi(x)\ dx=:I_1\varphi(r).
                \end{gathered}
            \end{equation*}\end{linenomath}
        \end{prop}
    \begin{proof}
    For the first equality,
        \begin{linenomath}\begin{equation*}
        \fform\mathscr{D}\fform^{-1}\varphi(r)=\fform\mathscr{D}[\varphi(t^2/2)](r)=-\fform[\varphi'(t^2/2)](r)=-\varphi'(r)=-\frac{d}{dr}\varphi(r),
        \end{equation*}\end{linenomath}
    and for the second equality,
    \begin{linenomath}\begin{equation*}
        \fform\mathscr{I}\fform^{-1}\psi(r)=\fform\left[\int_t^\infty x\psi(x^2/2)\ dx\right](r)=\fform\left[\int_{t^2/2}^\infty \psi(x)\ dx\right](r)=\int_r^\infty\psi(x)\ dx.
        \end{equation*}\end{linenomath}
    \end{proof}
    Proposition \ref{prop:I1_I-1} shows that the $f$-form operator simplifies the operators $\mathscr{D}$ and $\mathscr{I}$ into usual differential and integral operators. To cover more general cases, we can define the fractional integral operators
    \begin{linenomath}\begin{equation*}
        \begin{gathered}
            I_0\varphi(r)=\varphi(r),\\
            I_\nu\varphi(r):=\int_r^\infty\frac{(x-r)^{\nu-1}}{\Gamma(\nu)}\varphi(x)\ dx,\quad r\in\R_+,
            \end{gathered}
        \end{equation*}\end{linenomath}
    and the fractional differential operators
    \begin{linenomath}\begin{equation*}
        \begin{gathered}
            I_{-n}:=I_{-1}^n=(-1)^n\frac{d^n}{dr^n},\quad n\in\mathbb{N},\\
            I_{-\nu}:=I_{n-\nu}I_{-n},\quad 0<\nu\leq n=\ceil{\nu},
            \end{gathered}
        \end{equation*}\end{linenomath}
    for $\nu>0$.
    
    Next we shall find out the $f$-form of the $d$-variate Fourier transform $\mathscr{F}_d$ and the $d$-variate convolution $\mathscr{C}_d$. Schaback and Wu computed the $d$-variate Fourier transform of a given radial function \cite{SR96OP}. For a given function $\varphi$ on $\R_+$, they proceeded by first taking the inverse $f$-form of $\varphi$ to get 
    \begin{linenomath}
        \begin{equation*}
            g(x)=\fform^{-1}\varphi(\norm{x}{2})=\varphi(\norm{x}{2}^2/2),
            \end{equation*}
        \end{linenomath}
    and then computing the Fourier transform of $g$, which turns out to be
    \begin{linenomath}\begin{equation*}
        \hat{g}(\omega)=\int_0^\infty \varphi(r)r^{(d-2)/2}H_{(d-2)/2}(r\cdot\norm{\omega}{2}^2/2)\ dr=:F_{(d-2)/2}\varphi(\norm{\omega}{2}^2/2).
        \end{equation*}\end{linenomath}
    where $H_{(d-2)/2}$ is defined as in \eqref{defn:Jnu_Hnu}. 
    The function on the right side is exactly the inverse $f$-form of $F_{(d-2)/2}\varphi$ with variable $\norm{\omega}{2}$, so we write
    \begin{linenomath}\begin{equation*}
        \mathscr{F}_d\fform^{-1}\varphi(\norm{\omega}{2})=\hat{g}(\omega)=\fform^{-1}F_{(d-2)/2}\varphi(\norm{\omega}{2}).
        \end{equation*}\end{linenomath}
    Applying the $f$-form operator on both sides, we obtain the $f$-form of $\mathscr{F}_d$ by 
    \begin{linenomath}\begin{equation}
        F_{(d-2)/2}\varphi(r):=(\fform\mathscr{F}_d\fform^{-1}\varphi)(r)=\int_0^\infty\varphi(t)t^{(d-2)/2}H_{(d-2)/2}(tr)\ dt,\quad r\in\R_+,\ t^{(d-1)/2}f(t)\in L^1(\R_+).\label{defn:Fnu}
        \end{equation}\end{linenomath}
    Note that the operator $F_\nu$ defined in Eq.\eqref{defn:Fnu} is an integral operator with parameter $\nu$, so we can generalize $\nu$ to be any real numbers such that $d=2\nu+2>0$, where $d$ is a fractional dimension, whenever the integral exists. Similarly, the $f$-form of the convolution operator $\mathscr{C}_{d}$ should be
    \begin{linenomath}\begin{equation*}
         C_\nu(\varphi,\psi)(r):=(2\pi)^{-(\nu+1)}\fform\mathscr{C}_{2\nu+2}(\fform^{-1}\varphi,\fform^{-1}\psi)(r)=F_\nu[(F_\nu\varphi)\cdot(F_\nu\psi)]
        \end{equation*}\end{linenomath}
    for $\nu>-1$.
    
    Some important properties of the operators above are listed in Proposition \ref{prop:radial_trans}. Proofs of these properties are in \cite{SR96OP}.
    \begin{prop}\label{prop:radial_trans}
        Let $\varphi,\psi\in\mathscr{S}$ be tempered radial test functions.
        \begin{enumerate}[(1)]
            \item If $\mu>-1$, then $F_\mu^2\varphi=\varphi$.

            \item If $\mu,\nu\in\R$, then $I_\mu I_\nu\varphi=I_\nu I_\mu\varphi=I_{\mu+\nu}\varphi$.

            \item If $\mu,\nu>-1$, then $I_{\nu-\mu}\varphi=F_\mu F_\nu\varphi$ and $F_\mu\varphi=F_\nu I_{\mu-\nu}\varphi$.

            \item If $\mu>-1$, then $F_\mu [C_\mu(\varphi,\psi)]=(F_\mu\varphi)\cdot(F_\mu\psi)$.

            \item If $\mu,\nu>-1$, then $I_{\nu-\mu}C_\nu(\varphi,\psi)=C_\mu(I_{\nu-\mu}\varphi,I_{\nu-\mu}\psi)$.
            
            \end{enumerate}
        \end{prop}

    Now we can establish the dimension walk property of the operators $I_\nu$, where $\nu=\pm 1/2$ is not an integer. This is quite important because $\mathscr{I}$ and $\mathscr{D}$, or $I_1$ and $I_{-1}$ respectively, `walk' through dimensions by step $2$, while the following theorem implies that $I_{1/2}$ and $I_{-1/2}$ `walk' by step $1$. The proof is based on Proposition \ref{prop:radial_trans}(3) and the $f$-form operator.
    \begin{thm}\label{thm:gen_walk}
        Let $\varphi\in\mathscr{S}$ be a tempered radial test function. Let $\mathscr{D}^{1/2}:=\fform^{-1}I_{-1/2}\fform$ and $\mathscr{I}^{1/2}:=\fform^{-1}I_{1/2}\fform$. Then
        \begin{linenomath}
            \begin{equation*}
                \mathscr{F}_d\varphi=\mathscr{F}_{d+1}\mathscr{D}^{1/2}\varphi
            \end{equation*}
        \end{linenomath}
        and
        \begin{linenomath}
            \begin{equation*}
                \mathscr{F}_d\varphi=\mathscr{F}_{d-1}\mathscr{I}^{1/2}\varphi.
            \end{equation*}
        \end{linenomath}
    \end{thm}
    \begin{proof}
        Let $\psi:=\fform\varphi$ be the $f$-form of $\varphi$ and $\mu=(d-2)/2$, $\nu=(d-1)/2$. Then Proposition \ref{prop:radial_trans}(3) implies $F_{(d-2)/2}\psi=F_{(d-1)/2}I_{-1/2}\psi$, thus
        \begin{linenomath}
            \begin{equation*}
                \mathscr{F}_d\varphi=\fform^{-1}F_{(d-2)/2}\fform\varphi=\fform^{-1}F_{(d-2)/2}\psi=\fform^{-1}F_{(d-1)/2}I_{-1/2}\psi=\mathscr{F}_{d+1}(\fform^{-1}I_{-1/2})(\fform\varphi)=\mathscr{F}_{d+1}\mathscr{D}^{1/2}\varphi.
            \end{equation*}
        \end{linenomath}
        Another equality follows from $F_{(d-2)/2}\psi=F_{(d-3)/2}I_{1/2}\psi$ with $\mu=(d-2)/2$, $\nu=(d-3)/2$ in Proposition \ref{prop:radial_trans}(3) and similar computations.
    \end{proof}
    This theorem is parallel with Theorem \ref{thm:walk} and should be regarded as the generalized dimension walk property. It enables the fractional operators $I_\mu$ to `walk' through all space dimensions. Although Theorem \ref{thm:gen_walk} is stated for functions in $\mathscr{S}$, by standard continuity arguments, it is easy to extend this theorem to CSRBFs such that all the related terms exist.

% section 2
\section{The native spaces of Wu functions}

    \subsection{Reformulation of Wu functions}

    The original Wu functions are constructed as \cite{WZ95CO}
    \begin{linenomath}\begin{equation}\label{defn:wu_func}
        \varphi_{\ell,k}(r):=\mathscr{D}^k(f_\ell*f_\ell)(2r),\quad r\in\R_+,
        \end{equation}\end{linenomath}
    with $f_\ell(r)=(1-r^2)_+^\ell$ for $\ell,k\in\mathbb{N}$ such that $k\leq\ell$. The function $\varphi_{\ell,k}$ is a cut-off polynomial of degree $4\ell-2k+1$ with $2\ell-2k$ continuous derivatives, compactly supported on $[0,2]$ and positive definite on $\R^{2k+1}$ \cite{WZ95CO}.
    
    In the definition of Wu functions in \eqref{defn:wu_func}, the factor $2$ scales the function to the support $[0,1]$. But this factor will cause extra computations when applying multiple operators to this function. To be specific, for any $c>0$ and function $\varphi$ defined on $\R_+$, the action of operators on rescaled function $\varphi(cr)$ are
    \begin{linenomath}\begin{equation*}
        \begin{gathered}
            \fform\left[\varphi(ct)\right](r)=\varphi(\sqrt{2cr})=\fform\varphi(cr),\\
            \fform^{-1}\left[\varphi(ct)\right](r)=\varphi((cr)^2/2)=\fform^{-1}\varphi(cr),\\
            F_\nu[\varphi(ct)](r)=\int_0^\infty\varphi(cx)x^\nu H_\nu(rx)\ dx=c^{-\nu-1}F_\nu\varphi(r/c),
            \end{gathered}
        \end{equation*}\end{linenomath}
    for $\nu>-1$, and
    \begin{linenomath}\begin{equation*}
        \begin{gathered}
            I_\nu\left[\varphi(ct)\right](r)=\int_r^\infty\frac{(x-r)^{\nu-1}}{\Gamma(\nu)}\varphi(cx)\ dx=c^{-\nu}I_\nu\varphi(cr),\\
            I_{-n}\left[\varphi(ct)\right](r)=(-1)^n\frac{d^n}{dr^n}\varphi(cr)=c^nI_{-n}\varphi(cr),
            \end{gathered}
        \end{equation*}\end{linenomath}
    for $\nu,n>0$. To simplify computations, we remove the factor $2$ in the original definition. As the native spaces of Wu functions is one focus of this paper, we should be careful on whether the rescaling changes the native spaces of Wu functions. Since each Wu function is of finite smoothness, the result established by Larsson and Schaback indicates that the rescaling does not change the native spaces of Wu functions, namely the sets of functions are identical and the norms are equivalent \cite[Thm 1]{LE22SC}. In addition to removing the factor in the original definition of Wu functions, for a complete mathematical theory, we generalize $\ell$ to be any nonnegative real number. So we reformulate the Wu functions as follows.
    \begin{defn}
        For $k\in\mathbb{N},\ \ell\in\R_+$ and $k\leq\ell$, the Wu function $\varphi_{\ell,k}$ is defined by
        \begin{linenomath}
            \begin{equation}\label{defn:rescaled_wu}
                \varphi_{\ell,k}(r):=\mathscr{D}^k(f_\ell*f_\ell)(r),\quad r\in\R_+.
            \end{equation}
        \end{linenomath}
    \end{defn}
    
    Remember the question we proposed: Does a Wu function $\Phi(x)=\varphi_{\ell,k}(\norm{x}{2})$ reproduce a Sobolev space? And can we represent the native space of $\Phi$ in a simple way? To answer these two questions, we need to take a review on the general theory of \emph{reproducing kernels}.
    
\subsection{Preliminary on reproducing kernel theory}
    \begin{defn}(\cite[Sec. 1]{AN50TH}).
        Let $\mathcal{H}$ be a real Hilbert space of functions $f:\Omega\to\R$. A function $\Phi:\Omega\times\Omega\to\R$ is called a reproducing kernel for $\mathcal{H}$ if 
        \begin{enumerate}[(1)]
            \item $\Phi(\cdot,y)\in\mathcal{H}$ for all $y\in\Omega$,
            \item $f(y)=\innprod{f}{\Phi(\cdot,y)}_{\mathcal{H}}$ for all $f\in\mathcal{H}$ and all $y\in\Omega$.
        \end{enumerate}
     A \emph{Reproducing Kernel Hilbert Space (RKHS)} is a Hilbert space $\mathcal{H}$ with a reproducing kernel $\Phi$.
    \end{defn}
    \begin{defn}(\cite[Def. 4.3]{SR99NA}).
        If a symmetric (strictly) positive definite function $\Phi:\Omega\times\Omega\to\R$ is the reproducing kernel of a real Hilbert space $\mathcal{H}$ of real-valued functions on $\Omega$, then $\mathcal{H}$ is called the native space for $\Phi$. We denote the native space of $\Phi$ as $\mathcal{N}_\Phi(\Omega)$.
    \end{defn}
    There are several different ways to interpret the abstract elements in $\mathcal{N}_\Phi(\Omega)$ as functions. The following interpretation follows from \cite{SR99NA}, and we reorganize and summarize the results as the following properties.
    \begin{prop}
        Let $\Phi:\Omega\times\Omega\to\R$ be a symmetric positive definite kernel on $\Omega$. Then the native space of $\Phi$ has the following properties:
        \begin{enumerate}[(1)]
            \item $\mathcal{N}_{\Phi}(\Omega)\subset C(\Omega)$,
            \item $\mathcal{N}_{\Phi}(\Omega)\supset F_\Phi(\Omega)$, where $F_\Phi(\Omega)$ is an inner product space
            \begin{linenomath}
                \begin{equation*}
                    F_{\Phi}(\Omega):=\mbox{span}\{\Phi(\cdot,y):y\in\Omega\}
                \end{equation*}
            \end{linenomath}
            with inner product 
            \begin{linenomath}
                \begin{equation*}
                    \innprod{\sum_{j=1}^m\alpha_j\Phi(\cdot,x_j)}{\sum_{k=1}^n\beta_k\Phi(\cdot,y_k)}_{F_\Phi(\Omega)}:=\sum_{j=1}^m\sum_{k=1}^n\alpha_j\beta_k\Phi(x_j,y_k),
                \end{equation*}
            \end{linenomath}
            \item $\mathcal{N}_{\Phi}(\Omega)$ is the closure of $F_\Phi(\Omega)$ under the inner product of $F_\Phi(\Omega)$.
        \end{enumerate}
    \end{prop}
    The reader is referred to \cite{LL01TH} for other useful interpretations of abstract elements in native spaces.
    
    For two \emph{positive semi-definite kernels} $K_1(x,y)$ and $K_2(x,y)$, we say $K_1\ll K_2$,
    if $K_2-K_1$ is a positive semi-definite kernel. We say that these two kernels are equivalent and write 
    $ K_1\sim K_2,$if there exists positive constants $c_1,c_2>0$ such that $c_1K_1\ll K_2\ll c_2K_1$. This equivalence relation between kernels is an analogue of the equivalence relation between norms on a vector space. We say that two \emph{reproducing kernel Hilbert spaces (RKHSs)} $(H_1,\innprod{\cdot}{\cdot}_1)$ and $(H_2,\innprod{\cdot}{\cdot}_2)$ are equivalent and write 
    \begin{linenomath}\begin{equation*}
        (H_1,\innprod{\cdot}{\cdot}_1)\sim (H_2,\innprod{\cdot}{\cdot}_2),
        \end{equation*}\end{linenomath}
    if $H_1=H_2$ and there exists positive constants $c_1,c_2>0$ such that 
    \begin{linenomath}\begin{equation*}
        c_1\norm{\cdot}{1}\leq\norm{\cdot}{2}\leq c_2\norm{\cdot}{1}.
        \end{equation*}\end{linenomath}
    Aronszajn established the connection between kernels and their reproducing spaces \cite{AN50TH}.
    \begin{thm}\label{thm:ker2space}(\cite[Thm 7.I]{AN50TH}).
        If $K_i$ is the reproducing kernel of $(H_i,\innprod{\cdot}{\cdot}_i)$ for $i=1,2$ with $K_1\ll K_2$, then $H_1\subset H_2$ and $\norm{f}{1}\geq\norm{f}{2}$ for every $f\in H_1$.
        \end{thm}
    \begin{thm}\label{thm:space2ker}(\cite[Thm 7.II]{AN50TH}).
        If $K_2$ is the reproducing kernel of $(H_2,\innprod{\cdot}{\cdot}_2)$, $(H_1,\innprod{\cdot}{\cdot}_1)$ is a Hilbert space with $H_1\subset H_2$ and $\norm{f}{1}\geq\norm{f}{2}$ for every $f\in H_1$, then $(H_1,\innprod{\cdot}{\cdot}_1)$ possesses a reproducing kernel $K_1$ such that $K_1\ll K_2$.
    \end{thm}
    These two theorems implies that two kernels are equivalent if and only if their respective reproducing spaces are equivalent.
    \begin{thm}(\cite[1.13(C)]{AN50TH}).
        Let $(H_i,\innprod{\cdot}{\cdot}_i)$ be a RKHS with kernel $K_i$ for $i=1,2$. Then $K_1\sim K_2$ if and only if $(H_1,\innprod{\cdot}{\cdot}_1)\sim(H_2,\innprod{\cdot}{\cdot}_2)$.
        \end{thm}
    \begin{proof}
    The proof of this statement was sketched in Aronszajn's theory of reproducing kernels with sophisticated mathematics \cite[1.13(C)]{AN50TH}. Here we provide a simplified proof based on Theorems \ref{thm:ker2space} and \ref{thm:space2ker}.

    The sufficiency follows from Theorem \ref{thm:ker2space}. Since $K_1$ and $K_2$ are equivalent, there exists $c_1,c_2>0$ such that $c_1K_1\ll K_2\ll c_2K_1$. These two positive constants induce inner products on $H_1$ by 
    \begin{linenomath}\begin{equation*}
        \innprod{f}{g}_{1,c_1}:=c_1^{-1}\innprod{f}{g}_1,\quad\innprod{f}{g}_{1,c_2}:=c_2^{-1}\innprod{f}{g}_1.
        \end{equation*}\end{linenomath}
    The space $(H_1,\innprod{\cdot}{\cdot}_{1,c_j})$ is a RKHS with kernel $K_{1,c_j}=c_jK_1$ and norm $\norm{\cdot}{1,c_j}=c_j^{-1/2}\norm{\cdot}{1}$ according to the following direct computations.
    \begin{linenomath}\begin{equation*}
        \begin{gathered}
            \innprod{f}{K_{1,c_j}(\cdot,y)}_{1,c_j}=f(y)=c_j^{-1}\innprod{f}{c_jK_1(\cdot,y)}_1=\innprod{f}{c_jK_1(\cdot,y)}_{1,c_j},\\
            \norm{f}{1,c_j}=(\innprod{f}{f}_{1,c_j})^{1/2}=(c_j^{-1}\innprod{f}{f}_1)^{1/2}=c_j^{-1/2}\norm{f}{1}.
            \end{gathered}
        \end{equation*}\end{linenomath}
    So Theorem \ref{thm:ker2space} suggests that $H_1\subset H_2\subset H_1$ and $c_2^{-1/2}\norm{f}{1}\leq\norm{f}{2}\leq c_1^{-1/2}\norm{f}{1}$ for $f\in H_1$, which means exactly $(H_1,\innprod{\cdot}{\cdot}_1)\sim (H_2,\innprod{\cdot}{\cdot}_2)$.

    The necessity follows from Theorem \ref{thm:space2ker}. Since $c_1\norm{f}{1}\leq\norm{f}{2}\leq c_2\norm{f}{1}$ for $f\in H_1=H_2$, Theorem \ref{thm:space2ker} implies $c_2^{-2}K_1\ll K_2\ll c_1^{-2}K_1$ as $c_j^{-2}K_1$ is the kernel of $(H_1,c_j^2\innprod{\cdot}{\cdot}_1)$ with norm $c_j\norm{\cdot}{1}$.
    \end{proof}

    \subsection{The Sobolev spaces and the native spaces of Wu functions}
    
    Fourier transform is a powerful tool to characterize the positive definiteness of a function defined on Euclidean space \cite{BS33MO}\cite[Thm 6.6]{WH04SC}. Here it could be used to decide whether two given kernels induced by radial functions are equivalent according to Bochner's characterization \cite{BS33MO}, and thus decide whether the respective RKHSs are equivalent. In this subsection, we will focus on the kernels induced by Wu functions and the kernels of Sobolev spaces. 
    
    Let $k$ be a nonnegative integer and $d=2k+1$ be the dimension of space. Suppose $\Phi(x)=\varphi_{\ell,k}(\norm{x}{2})$ is the radial function on $\R^d$ induced by the Wu function $\varphi_{\ell,k}$. To see weather the native space of $\Phi$ is equivalent to a Sobolev space, we need to compute the Fourier transform of the kernels of Sobolev spaces and that of $\Phi$. For dimension $d$ and $s>0$, the Sobolev space $H^s(\R^d)$ is defined by
    \begin{linenomath}\begin{equation}
        H^s(\R^d)=\left\{f\in L^2(\R^d): \hat{f}(\cdot)(1+\norm{\cdot}{2}^2)^{s/2}\in L^2(\R^d)\right\}, \label{eq:Hsd}
        \end{equation}\end{linenomath}
    with an inner product
    \begin{linenomath}\begin{equation*}
        \innprod{f}{g}_{H^s(\R^d)}:=(2\pi)^{-d/2}\int_{\R^d}\hat{f}(\omega)\overline{\hat{g}(\omega)}(1+\norm{\omega}{2}^2)^s\ d\omega.
        \end{equation*}\end{linenomath}
    Formal computations in \cite{WH04SC} show that for $s>d/2$ the Sobolev space $H^s(\R^d)$ is the native space of a positive definite radial function
    \begin{linenomath}\begin{equation*}
        \Psi(x)=\frac{2^{1-s}}{\Gamma(s)}\norm{x}{2}^{s-d/2}K_{d/2-s}(\norm{x}{2})=:\psi(\norm{x}{2}),
        \end{equation*}\end{linenomath}
    where $K_\nu$ is the \emph{modified Bessel function of the third kind} \cite[9.6]{AM64HA}. The Fourier transform of $\Psi$ is \emph{an inverse multiquadric} given by
    \begin{linenomath}\begin{equation}
        \hat{\Psi}(\omega)=\mathscr{F}_d\psi(\norm{\omega}{2})=(1+\norm{\omega}{2}^2)^{-s}.\label{eq:imq}
        \end{equation}\end{linenomath}
    Hence, the kernels of Sobolev spaces are characterized by radial functions with Fourier transform that has no zero and decays algebraically. Next we shall consider the $d$-variate Fourier transform of the Wu function $\varphi_{\ell,k}$.
    \begin{lem}\label{lem:fourier_Hnu}(\cite[Lem. 2.1]{SR96OP}).
        For $\nu>\mu>-1$ and $c,r>0$, we have
        \begin{linenomath}
            \begin{equation*}
                F_\mu[H_\nu(ct)](r)=\frac{1}{\Gamma(\nu-\mu)}c^{-\nu}(c-r)_+^{\nu-\mu-1},
            \end{equation*}
        \end{linenomath}
        where $H_\nu(x^2/4)=(x/2)^{-\nu}J_\nu(x)$ and $J_\nu$ is the Bessel function of the first kind and $F_{\mu}$ is defined in \eqref{defn:Fnu}.
    \end{lem}
    \begin{thm}\label{thm:fourier_wu}
        Suppose $k\in\mathbb{N},\ \ell\in\R_+$ and $k\leq\ell$. Let $d=2k+1$. Then the $d$-variate Fourier transform of the Wu function $\varphi_{\ell,k}$ is given by
    \begin{linenomath}\begin{equation*}
        \mathscr{F}_d\varphi_{\ell,k}(r)=\frac{\Gamma(\ell+1)^2}{2}\sqrt{2\pi}H_{\ell+1/2}^2(r^2/4).
        \end{equation*}\end{linenomath}
    \end{thm}
    \begin{proof}
        This result is given in \cite[Sec.7]{SR96OP}. Here we introduce a simpler derivation based on the properties of the $f$-form operator. 
        
        Applying the $f$-form operator on $\varphi_{\ell,k}$, Proposition \ref{prop:radial_trans}(5) implies that
    \begin{linenomath}\begin{equation*}
        \begin{aligned}
            \fform\varphi_{\ell,k}=\fform\mathscr{D}^k\mathscr{C}_1(f_\ell,f_\ell)&=\sqrt{2\pi}I_{-k}C_{-1/2}(\fform f_\ell,\fform f_\ell)\\
            &=\sqrt{2\pi}C_{k-1/2}(I_{-k}\fform f_\ell,I_{-k}\fform f_\ell).
            \end{aligned}
        \end{equation*}\end{linenomath}
        The $f$-form of $f_\ell(r)=(1-r^2)_+^\ell$ is a truncated power function
    \begin{linenomath}\begin{equation*}
        \fform f_\ell(r)=(1-2r)_+^\ell.
        \end{equation*}\end{linenomath}
        The $k$-th derivative of $\fform f_\ell$ is
    \begin{linenomath}\begin{equation*}
        I_{-k}\fform f_\ell(r)=\frac{\ell!}{(\ell-k)!}2^k\cdot(1-2r)_+^{\ell-k}=\frac{\Gamma(\ell+1)}{\Gamma(\ell-k+1)}2^k\cdot\fform f_{\ell-k}(r).
        \end{equation*}\end{linenomath}
    Applying $F_\mu$ on the equality in Lemma \ref{lem:fourier_Hnu}, we have
    \begin{linenomath}
        \begin{equation*}
        H_\nu(ct)=\frac{c^{-\nu}}{\Gamma(\nu-\mu)}F_\mu[(c-r)_+^{\nu-\mu-1}](t),
        \end{equation*}
    \end{linenomath}
    so by inserting $c=1/2$ we have
    \begin{linenomath}
        \begin{equation*}
        F_\mu\fform f_{\nu-\mu-1}(r)=F_\mu[(1-2r)_+^{\nu-\mu-1}](r)=\frac{\Gamma(\nu-\mu)}{2^{\mu+1}}H_\nu(r/2),
        \end{equation*}
    \end{linenomath}
    therefore
    \begin{linenomath}
        \begin{equation*}
        F_{\mu}\fform f_{\nu}(r)=\frac{\Gamma(\nu+1)}{2^{\mu+1}}H_{\nu+\mu+1}(r/2).
        \end{equation*}
    \end{linenomath}
    Concluding all the computations above, the action of $F_{k-1/2}$ on $\fform\varphi_{\ell,k}(r)$ is
    \begin{linenomath}\begin{equation*}
        \begin{aligned}
            F_{k-1/2}\fform\varphi_{\ell,k}(r)&=\sqrt{2\pi}F_{k-1/2}C_{k-1/2}(I_{-k}\fform f_\ell,I_{-k}\fform f_\ell)(r)\\
            &=\sqrt{2\pi}[F_{k-1/2}(I_{-k}\fform f_\ell)(r)]^2\\
            &=\frac{\Gamma(\ell+1)^2}{\Gamma(\ell-k+1)^2}4^k\sqrt{2\pi}[F_{k-1/2}\fform f_{\ell-k}(r)]^2\\
            &=\frac{\Gamma(\ell+1)^2}{2}\sqrt{2\pi}H_{\ell+1/2}^2(r/2),
            \end{aligned}
        \end{equation*}\end{linenomath}
    and thus the $2k+1$-variate Fourier transform of $\varphi_{\ell,k}$ is
    \begin{equation}\label{eq:fourier_wu}
        \mathscr{F}_{2k+1}\varphi_{\ell,k}(r)=\fform^{-1}F_{k-1/2}\fform\varphi_{\ell,k}(r)=\frac{\Gamma(\ell+1)^2}{2}\sqrt{2\pi}H_{\ell+1/2}^2(r^2/4).
        \end{equation}
    \end{proof}

    Now we are ready to examine the inclusion relation between the native space of $\Phi(x)=\varphi_{\ell,k}(\norm{x}{2})$ and the Sobolev spaces $H^s(\R^d)$ on odd dimensions $d=2k+1$ with $s>d/2$. An interesting fact related to the functions $H_{\ell+1/2}(r^2/4)=(r/2)^{-(\ell+1/2)}J_{\ell+1/2}(r)$ is that these functions possess zeros in $\R_+$. For sufficiently large $r$, an approximation of $J_{\ell+1/2}$ in \cite[9.2.1]{AM64HA} is
    \begin{linenomath}
        \begin{equation}\label{eq:approx_jnu}
            \sqrt{\frac{\pi r}{2}}J_{\ell+1/2}(r)=\cos\left(r-\frac{(\ell+1/2)\pi}{2}-\frac{\pi}{4}\right)+O(r^{-1/2}),
        \end{equation}
    \end{linenomath}
    so the large zeros of $J_{\ell+1/2}$ are asymptotically
    \begin{linenomath}
        \begin{equation*}
            \left(\frac{\ell}{2}+k+1\right)\pi,\quad k\in\N.
        \end{equation*}
    \end{linenomath}
    Some small positive zeros of $J_\nu$, computed with the Mathematica function \texttt{BesselJZero}, are listed in Table \ref{table:zeros_jnu}. Since $H_{\ell+1/2}$ has zeros while the inverse multiquadrics have no zero, the inequality
    \begin{equation}
        \mathscr{F}_d\psi\leq c\cdot\mathscr{F}_d\varphi_{\ell,k}
        \end{equation}
    fails for every $c>0$ on the zeros of $\mathscr{F}_d\varphi_{\ell,k}$, and thus the native space of $\Phi$ could not be equivalent with any Sobolev space $H^s(\R^d)$ with $s>d/2$. For the inverse inequality, \eqref{eq:approx_jnu} implies that the Bessel function $J_\nu(r)$ is bounded by the inverse square-root for large $r$, so
    \begin{linenomath}\begin{equation*}
        H_{\ell+1/2}^2(r^2/4)=(r/2)^{-(2\ell+1)}J_{\ell+1/2}^2(r)=O(r^{-(2\ell+2)}),\quad r\to\infty,
        \end{equation*}\end{linenomath}
    and the inverse inequality
    \begin{equation}\label{eq:wu<sobolev}
        \mathscr{F}_d\varphi_{\ell,k}\leq c_{\ell,k}\cdot\mathscr{F}_d\psi
        \end{equation}
    holds for $k+1/2<s\leq\ell+1$ and some $c_{\ell,k}>0$. We conclude all these results into the following theorem.
    
    \begin{thm}\label{thm:wu_subspace_sobolev}
        Let $\varphi_{\ell,k}$ be a Wu function with $k\in\mathbb{N},\ \ell\in\R_+$ and $k\leq \ell$. Define a radial function $\Phi(x)=\varphi_{\ell,k}(\norm{x}{2})$ on $\R^{2k+1}$. Then $\mathcal{N}_\Phi(\R^{2k+1})$ is a subspace of $H^s(\R^{2k+1})$ for $k+1/2<s\leq\ell+1$.
        \end{thm}
    \begin{proof}
        The inequality \eqref{eq:wu<sobolev} indicates the inclusion $\mathcal{N}_\Phi(\R^{2k+1})\subset H^s(\R^{2k+1})$ and the norm inequality
        \begin{linenomath}\begin{equation*}
            \norm{\cdot}{\mathcal{N}_\Phi(\R^{2k+1})}\geq c^{-1/2}\norm{\cdot}{H^s(\R^{2k+1})}
            \end{equation*}\end{linenomath}
        for $k+1/2<s\leq\ell+1$ and some $c>0$, according to Theorem \ref{thm:ker2space}, and the norm inequality indicates that the inclusion map is continuous.
        \end{proof}
    Although Theorem \ref{thm:wu_subspace_sobolev} shows that $\mathcal{N}_\Phi(\R^{2k+1})$ is a subspace of some Sobolev spaces, the orthogonal complements of $\mathcal{N}_\Phi(\R^{2k+1})$ in these Sobolev spaces are still unclear. We will discuss this question later.

    \subsection{Representation of the native spaces of Wu functions}

    In this subsection, we will derive a convolutional representation for the native space $\mathcal{N}_\Phi(\R^{2k+1})$, where $\Phi(x)=\varphi_{\ell,k}(\norm{x}{2})$ and $\varphi_{\ell,k}$ is a Wu function. It turns out that the native space of a Wu function is closely connected to its `convolution square-root'. Note that similar result under different assumptions on the kernel has been established by Schaback \cite{SR00AU}.
    
    Wendland characterized the native space of a given function in the frequency domain \cite{WH04SC}.
    \begin{lem}\label{lem:native_space_phase}(\cite[Thm 10.12]{WH04SC}).
        Suppose $\Phi\in C(\R^d)\ints L^1(\R^d)$ is a real-valued positive definite function. Define
        \begin{linenomath}
            $$\mathcal{N}_\Phi(\R^d):=\{f\in C(\R^d)\ints L^2(\R^d):\hat{f}/\sqrt{\hat{\Phi}}\in L^2(\R^d)\},$$
            \end{linenomath}
        and equip this space with the inner product
        \begin{linenomath}\begin{equation*}
            \innprod{f}{g}_{\mathcal{N}_\Phi(\R^d)}:=(2\pi)^{-d/2}\innprod{\hat{f}/\sqrt{\hat{\Phi}}}{\hat{g}/\sqrt{\hat{\Phi}}}_{L^2(\R^d)}=(2\pi)^{-d/2}\int_{\R^d}\frac{\hat{f}(\omega)\overline{\hat{g}(\omega)}}{\hat{\Phi}(\omega)}\ d\omega.
            \end{equation*}\end{linenomath}
        Then $\mathcal{N}_\Phi(\R^d)$ is a real Hilbert space with inner product $\innprod{\cdot}{\cdot}_{\mathcal{N}_\Phi(\R^d)}$ and reproducing kernel $\Phi(\cdot-\cdot)$. Hence $\mathcal{N}_\Phi(\R^d)$ is the native space of $\Phi$ on $\R^d$.
        \end{lem}

    According to Lemma \ref{lem:native_space_phase}, a function $f$ is in the native space of $\Phi$ if and only if the function $g=\hat{f}/\sqrt{\hat{\Phi}}$ is square integrable. Thus if $g$ is a square integrable function, then the function $f=(g\sqrt{\hat{\Phi}})^\vee$ should always be in $\mathcal{N}_\Phi(\R^d)$ because 
    \begin{linenomath}
        \begin{equation*}
            \int_{\R^d}\frac{\abs{\hat{f}(\omega)}^2}{\abs{\sqrt{\hat{\Phi}(\omega)}}^2}\ d\omega=\int_{\R^d}\frac{\abs{g(\omega)\sqrt{\hat{\Phi}(\omega)}}^2}{\abs{\hat{\Phi}(\omega)}}\ d\omega\leq\int_{\R^d}\abs{g(\omega)}^2\ d\omega<\infty,
        \end{equation*}
    \end{linenomath}
    that is, $\hat{f}/\sqrt{\hat{\Phi}}$ is square integrable. Furthermore, the convolution theorem \cite[Thm 5.16(2)]{WH04SC} implies
    \begin{linenomath}\begin{equation*}
        (g\sqrt{\hat{\Phi}})^\vee=(2\pi)^{-d/2}[g^\vee*(\sqrt{\hat{\Phi}})^\vee].
        \end{equation*}\end{linenomath}
    This allows a convolutional representation for $\mathcal{N}_\Phi(\R^d)$ under certain assumptions on $\Phi$.
    \begin{thm}\label{thm:native_space_conv}
        Suppose $\Phi=m*m$ with $m\in L^1(\R^d)\ints L^2(\R^d)$ such that $\hat{m}$ vanishes on a zero measure set. Then $\Phi\in C(\R^d)\ints L^1(\R^d)$ is positive definite, and the native space of $\Phi$ on $\R^d$ is
        \begin{linenomath}\begin{equation*}
            \mathcal{N}_\Phi(\R^d)=m*L^2(\R^d):=\{m*g:g\in L^2(\R^d)\}.
            \end{equation*}\end{linenomath}
        Moreover, the convolution operator
        \begin{linenomath}\begin{equation*}
            C_m:L^2(\R^d)\to\mathcal{N}_\Phi(\R^d),\quad C_mg:=m*g
            \end{equation*}\end{linenomath}
        is a bijective isometry.
        \end{thm}
    \begin{proof}
        Young's inequality \cite[Cor. 1.21]{DJ01FO} implies $\Phi\in L^1(\R^d)$. That $\Phi\in C(\R^d)$ follows from Young's inequality and the fact \cite[Lem 5.21]{WH04SC} that
        \begin{linenomath}\begin{equation*}
            \lim_{t\to 0}\norm{f(\cdot-t)-f(\cdot)}{L^1(\R^d)}=0,\quad f\in L^1(\R^d).
            \end{equation*}\end{linenomath}
        Since $\Phi$ is the self convolution of $m$, the Fourier transform of $\Phi$ is
        \begin{linenomath}\begin{equation*}
            \hat{\Phi}=(2\pi)^{d/2}\hat{m}^2\geq 0,
            \end{equation*}\end{linenomath}
        so $\Phi$ is positive semi-definite by Bochner's characterization. Since $\hat{m}$ vanishes on a zero measure set and $\hat{\Phi}$ is continuous, the support of $\hat{\Phi}$ contains an open set, and thus $\Phi$ is positive definite \cite[Thm 6.8]{WH04SC}.
        
        For every $g\in L^2(\R^d)$, the norm of $C_mg$ in $\mathcal{N}_\Phi(\R^d)$ is
        \begin{linenomath}\begin{equation*}
            \begin{aligned}
                \norm{C_mg}{\mathcal{N}_\Phi(\R^d)}&=\left[(2\pi)^{-d}\int_{\R^d}\frac{\abs{(C_mg)^\wedge(\omega)}^2}{\abs{\hat{m}(\omega)}^2}\ d\omega\right]^{1/2}=\left[\int_{\R^d}\abs{\hat{g}(\omega)}^2\ d\omega\right]^{1/2}\\
                &=\norm{\hat{g}}{L^2(\R^d)}=\norm{g}{L^2(\R^d)},
                \end{aligned}
            \end{equation*}\end{linenomath}
        where the last equality follows from Plancherel's theorem \cite[Cor. 5.25]{WH04SC}, so $C_m$ is an isometry that maps $L^2(\R^d)$ into $\mathcal{N}_\Phi(\R^d)$. For every $f\in\mathcal{N}_\Phi(\R^d)$, the function $g:=(2\pi)^{-d/2}(\hat{f}/\hat{m})^\vee$ is square integrable, and
        \begin{linenomath}\begin{equation*}
            C_mg=m*g=(2\pi)^{d/2}(\hat{m}\hat{g})^\vee=(\hat{f})^\vee=f,
            \end{equation*}\end{linenomath}
        so $C_m$ is surjective and thus bijective. As a result, the representation
        \begin{linenomath}\begin{equation*}
            \mathcal{N}_\Phi(\R^d)=C_m(L^2(\R^d))=\{m*g:g\in L^2(\R^d)\}
            \end{equation*}\end{linenomath}
        follows.
        \end{proof}
    
    For a Wu function $\varphi_{\ell,k}$, the radial function $\Phi(x)=\varphi_{\ell,k}(\norm{x}{2})$ on $\R^d$ is the self convolution of some radial function, as we will see later. Applying Theorem \ref{thm:native_space_conv} on Wu functions, we get the following result.

    \begin{thm}\label{thm:native_space_wu}
        Let $\varphi_{\ell,k}$ be a Wu function with $k\in\mathbb{N},\ \ell\in\R_+$ and $k\leq\ell$. Then the native space of $\Phi(x)=\varphi_{\ell,k}(\norm{x}{2})$ on $\R^{2k+1}$ is
        \begin{equation}
            \mathcal{N}_\Phi(\R^{2k+1})=P*L^2(\R^{2k+1}):=\{P*g:g\in L^2(\R^{2k+1})\},
            \end{equation}
        where
        \begin{linenomath}\begin{equation*}
            \begin{gathered}
                P(x)=\frac{\Gamma(\ell+1)}{\Gamma(\ell-k+1)}(2/\pi)^{k/2}f_{\ell-k}(\norm{x}{2}),\\
                f_{\ell-k}(r)=(1-r^2)_+^{\ell-k}.
                \end{gathered}
        \end{equation*}\end{linenomath}
        Moreover, the convolution operator
        \begin{linenomath}\begin{equation*}
            C_P:L^2(\R^{2k+1})\to\mathcal{N}_\Phi(\R^{2k+1}),\quad C_Pg=P*g
            \end{equation*}\end{linenomath}
        is a bijective isometry.
        \end{thm}
        
    \begin{proof}
        We break $\Phi$ into the self convolution of some function. With Proposition \ref{prop:radial_trans}(4,5), compute
        \begin{equation}
            \begin{aligned}
                \mathscr{F}_{2k+1}\varphi_{\ell,k}&=\mathscr{F}_{2k+1}\mathscr{D}^k(f_\ell*f_\ell)(r)\\
                &=\sqrt{2\pi}\fform^{-1}F_{k-1/2}I_{-k}C_{-1/2}(\fform f_{\ell},\fform f_{\ell})\\
                &=\sqrt{2\pi}\fform^{-1}F_{k-1/2}C_{k-1/2}(I_{-k}\fform f_{\ell},I_{-k}\fform f_{\ell})\\
                &=\sqrt{2\pi}(\fform^{-1}F_{k-1/2}I_{-k}\fform f_{\ell})^2.
                \end{aligned}
            \end{equation}
        
        Notice that
        \begin{linenomath}\begin{equation*}
            I_{-k}\fform f_\ell(r)=\left(-\frac{d}{dr}\right)^{k}(1-2r)_+^\ell=\frac{\Gamma(\ell+1)}{\Gamma(\ell-k+1)}2^k\fform f_{\ell-k}(r),
            \end{equation*}\end{linenomath}
        so
        \begin{linenomath}\begin{equation*}
            \begin{aligned}
                \mathscr{F}_{2k+1}\varphi_{\ell,k}&=\frac{\Gamma(\ell+1)^2}{\Gamma(\ell-k+1)^2}2^{2k}\sqrt{2\pi}(\fform^{-1}F_{k-1/2}\fform f_{\ell-k})^2\\
                &=\frac{\Gamma(\ell+1)^2}{\Gamma(\ell-k+1)^2}2^{2k}\sqrt{2\pi}(\mathscr{F}_{2k+1}f_{\ell-k})^2,\\
                &=\frac{\Gamma(\ell+1)^2}{\Gamma(\ell-k+1)^2}(2/\pi)^k\mathscr{F}_{2k+1}\mathscr{C}_{2k+1}(f_{\ell-k},f_{\ell-k})\\
                &=\mathscr{F}_{2k+1}\mathscr{C}_{2k+1}(p,p),
                \end{aligned}
            \end{equation*}\end{linenomath}
        where 
        \begin{linenomath}\begin{equation*}
            p(r)=\frac{\Gamma(\ell+1)}{\Gamma(\ell-k+1)}(2/\pi)^{k/2}f_{\ell-k}(r).
            \end{equation*}\end{linenomath}
        This means that
        \begin{equation}
            \hat{\Phi}(\omega)=(P*P)^\wedge(\omega),
            \end{equation}
        where $P(x)=p(\norm{x}{2})$, so $\Phi$ is the self convolution of $P$. The remaining facts follow directly from Lemma \ref{lem:native_space_phase}.
        \end{proof}
    
    Theorem \ref{thm:native_space_wu} provides an explicit characterization for the native spaces of Wu functions. For $\Phi(x)=\varphi_{\ell,k}(\norm{x}{2})$, a function in $\mathcal{N}_\Phi(\R^{2k+1})$ is, in the space domain, the convolution of a $L^2$ function and the `convolution square-root' of $\Phi$ or, in the phase domain, a $L^2$ signal modulated by the weight function $H_{\ell+1/2}(\norm{\omega}{2}^2/4)$. Since $H_{\ell+1/2}$ has zeros in $\R_+$, the functions in $\mathcal{N}_\Phi(\R^{2k+1})$ may miss certain frequencies at the zeros of $H_{\ell+1/2}(\norm{\omega}{2}^2/4)$. But the following theorem shows that the missing of such frequencies is in fact not a serious problem.
    \begin{thm}\label{thm:dense}
        Let $\varphi_{\ell,k}$ be a Wu function with $k\in\mathbb{N},\ \ell\in\R_+$ and $k\leq \ell$. Then the native space of $\Phi(x)=\varphi_{\ell,k}(\norm{x}{2})$ on $\R^{2k+1}$ is dense in $H^s(\R^{2k+1})$ for $k+1/2<s\leq\ell+1$.
        \end{thm}
    \begin{proof}
        Suppose $g\in H^s(\R^{2k+1})$ is orthogonal with all the functions in $\mathcal{N}_\Phi(\R^{2k+1})$. Then for $h\in L^2(\R^{2k+1})$,
        \begin{equation}\label{eq:orthogonal}
            \begin{aligned}
                0=\innprod{g}{P*h}_{H^s(\R^{2k+1})}&=(2\pi)^{-(k+1/2)}\int_{\R^{2k+1}}\hat{g}(\omega)\overline{(P*h)^\wedge(\omega)}(1+\norm{\omega}{2}^2)^s\ d\omega\\
                &=\int_{\R^{2k+1}}\hat{g}(\omega)\hat{P}(\omega)(1+\norm{\omega}{2}^2)^s\overline{\hat{h}(\omega)}\ d\omega\\
                &=\innprod{\hat{g}(\omega)\hat{P}(\omega)(1+\norm{\omega}{2}^2)^s}{\hat{h}(\omega)}_{L^2(\R^{2k+1})}.
                \end{aligned}
            \end{equation}
        Since $\hat{g}\cdot(1+\norm{\omega}{2}^2)^s$ is square integrable and $\hat{P}$ is bounded, the function $\hat{g}\hat{P}\cdot(1+\norm{\omega}{2}^2)^s$ is square integrable. Equation \eqref{eq:orthogonal} implies that the function $\hat{g}\hat{P}\cdot(1+\norm{\omega}{2}^2)^s$ should be orthogonal with all functions in $L^2(\R^{2k+1})$, so this function should be the zero function. But $\hat{P}\cdot(1+\norm{\omega}{2}^2)^s$ vanishes on a zero measure set, so $g$ must be zero. In conclusion, the orthogonal complement of $\mathcal{N}_\Phi(\R^{2k+1})$ in $H^s(\R^{2k+1})$ contains only the zero function, and $\mathcal{N}_\Phi(\R^{2k+1})$ is dense in $H^s(\R^{2k+1})$ because
        \begin{linenomath}\begin{equation*}
            \overline{\mathcal{N}_\Phi(\R^{2k+1})}=(\mathcal{N}_\Phi(\R^{2k+1})^\perp)^{\perp}=0^\perp=H^s(\R^{2k+1}).
            \end{equation*}\end{linenomath}
        \end{proof}

\section{The generalized Wu functions and their closed form}
    
\subsection{The generalized Wu functions}
    Up to now we have obtained three equalities for Wu functions in odd-dimensional spaces, namely
    \begin{align}
        \varphi_{\ell,k}(r)&=\mathscr{D}^k(f_\ell*f_\ell)(r) \label{defn:wu1}\\
        &=c_{\ell,k}\mathscr{C}_{2k+1}(f_{\ell-k},f_{\ell-k})(r)\label{defn:wu2}\\
        &=d_{\ell}\mathscr{F}_{2k+1}[H_{\ell+1/2}^2(t^2/4)](r), \label{defn:wu3}
        \end{align}
    where the constants are 
    \begin{linenomath}\begin{equation*}
        c_{\ell,k}=\frac{\Gamma(\ell+1)^2}{\Gamma(\ell-k+1)^2}(2/\pi)^k,\quad d_\ell=\frac{\Gamma(\ell+1)^2}{2}\sqrt{2\pi}.
        \end{equation*}\end{linenomath}
    These equalities reveal different features of Wu functions. Eq. \eqref{defn:wu1}, the original definition for Wu function, concerns the dimension walk property of the differential operator $\mathscr{D}$. Eq. \eqref{defn:wu2}, proved in Theorem \ref{thm:native_space_wu}, represents each Wu function as the self convolution of a certain function, which appears as the convolution kernel in the representation of the native space of a given Wu function. Eq. \eqref{defn:wu3}, computed in Theorem \ref{thm:fourier_wu}, shows the $2k+1$-dimensional Fourier transform of Wu functions. So to extend the family of Wu functions, one needs to keep the consistency of these three equalities. In this subsection, we will verify that the equalities \eqref{defn:wu1}, \eqref{defn:wu2} and \eqref{defn:wu3} hold for nonnegative half integers $k$, and the family of Wu functions extends to all integer-dimensional spaces as a result.
    \begin{thm}\label{thm:gen_wu_func}
        Let $\fform$ be the $f$-form operator, $I_{-k}$ be the fractional differential operator. Define $\mathscr{D}^k$ by
        \begin{linenomath}\begin{equation*}
            \mathscr{D}^k:=\fform^{-1}I_{-k}\fform,\quad k\in\mathbb{N}/2.
            \end{equation*}\end{linenomath}
        Then the equalities 
        \begin{linenomath}\begin{equation*}
            \mathscr{D}^k(f_\ell*f_\ell)=c_{\ell,k}\mathscr{C}_{2k+1}(f_{\ell-k},f_{\ell-k})=d_\ell\mathscr{F}_{2k+1}[H_{\ell+1/2}^2(t^2/4)]
            \end{equation*}\end{linenomath}
        hold for $k\in\mathbb{N}/2,\ \ell\in\R_+$ such that $\ell\geq k$, where
        \begin{linenomath}\begin{equation*}
            \begin{gathered}
                f_\ell(r)=(1-r^2)_+^\ell,\\
                c_{\ell,k}=\frac{\Gamma(\ell+1)^2}{\Gamma(\ell-k+1)^2}(2/\pi)^k,\quad d_\ell=\frac{\Gamma(\ell+1)^2}{2}\sqrt{2\pi}.
                \end{gathered}
            \end{equation*}\end{linenomath}
        \end{thm}
    \begin{proof}
        For convenience we write $g_\ell(r)=\fform f_\ell(r)=(1-2r)_+^\ell$. The equalities to be verified are equivalent to
        \begin{linenomath}\begin{equation*}
            (2\pi)^{1/2}I_{-k}C_{-1/2}(g_\ell,g_\ell)=a_{\ell,k}C_{k-1/2}(g_{\ell-k},g_{\ell-k})=d_\ell F_{k-1/2}[H_{\ell+1/2}^2(t/2)],
            \end{equation*}\end{linenomath}
        where the constant $a_{\ell,k}=(2\pi)^{k+1/2}c_{\ell,k}$.
        
        For the first equality, Proposition \ref{prop:radial_trans}(5) indicates that
        \begin{equation}\label{eq:swap_I_C}
            I_{-k}C_{-1/2}(g_\ell,g_\ell)=C_{k-1/2}(I_{-k}g_\ell,I_{-k}g_\ell).
            \end{equation}
        The fractional derivative of $g_\ell$, a truncated power function, is also a truncated power function, which is
        \begin{linenomath}\begin{equation*}
            I_{-k}g_\ell(r)=2^k\frac{\Gamma(\ell+1)}{\Gamma(\ell-k+1)}g_{\ell-k}(r),\quad k\in\mathbb{N},
            \end{equation*}\end{linenomath}
        and 
        \begin{linenomath}
        \begin{equation*}
            \begin{aligned}
                I_{-k}g_\ell(r)&=I_{-k-1/2}I_{1/2}g_\ell(r)\\
                &=I_{-k-1/2}\int_r^{1/2}\frac{(x-r)^{-1/2}}{\Gamma(1/2)}(1-2x)_+^\ell\ dx\\
                &=I_{-k-1/2}\left[2^{-1/2}\frac{B(1/2,\ell+1)}{\Gamma(1/2)}(1-2r)_+^{\ell+1/2}\right]\\
                &=2^{-1/2}\frac{B(1/2,\ell+1)}{\Gamma(1/2)}I_{-k-1/2}g_{\ell+1/2}(r)\\
                &=2^k\frac{\Gamma(\ell+1)}{\Gamma(\ell-k+1)}g_{\ell-k}(r),\quad k\in\mathbb{N}/2,
                \end{aligned}
            \end{equation*}
            \end{linenomath}
        where $B(\cdot,\cdot)$ is the Beta function \cite[6.2.1]{AM64HA}. So we have
        \begin{linenomath}\begin{equation*}
            \begin{aligned}
                (2\pi)^{1/2}I_{-k}C_{-1/2}(g_\ell,g_\ell)&=\frac{\Gamma(\ell+1)^2}{\Gamma(\ell-k+1)^2}2^{2k}(2\pi)^{1/2}C_{k-1/2}(I_{-k}g_\ell,I_{-k}g_\ell)\\
                &=a_{\ell,k}C_{k-1/2}(g_{\ell-k},g_{\ell-k}).
                \end{aligned}
            \end{equation*}\end{linenomath}
        For the second equality, by applying $F_{k-1/2}$ on Eq. \eqref{eq:swap_I_C} we have
        \begin{linenomath}\begin{equation*}
            \begin{aligned}
                F_{k-1/2}I_{-k}C_{-1/2}(g_\ell,g_\ell)&=(F_{k-1/2}I_{-k}g_\ell)^2=(F_{-1/2}g_\ell)^2\\
                &=\frac{\Gamma(\ell+1)^2}{2}H_{\ell+1/2}^2(r/2),
                \end{aligned}
            \end{equation*}\end{linenomath}
        therefore
        \begin{linenomath}\begin{equation*}
            (2\pi)^{1/2}I_{-k}C_{-1/2}(g_\ell,g_\ell)=d_\ell F_{k-1/2}[H_{\ell+1/2}^2(t/2)].
            \end{equation*}\end{linenomath}
        \end{proof}
    Theorem \ref{thm:gen_wu_func} allows us to define Wu functions on any integer-dimensional space with equalities \eqref{defn:wu1}, \eqref{defn:wu2} and \eqref{defn:wu3}. The functions on odd-dimensional spaces are the rescaled Wu functions, and those on even-dimensional spaces are the rescaled \emph{missing Wu functions}. 
    \begin{rem}
        The computations in Theorem \ref{thm:gen_wu_func} hold naturally for all nonnegative real number $k\in\R_+$, so the Wu functions could be defined on `fractional'-dimensional spaces. We shall call such functions 
        \begin{linenomath}\begin{equation*}
            \varphi_{\ell,k},\quad k,\ell\in\R_+,\ k\leq\ell,
            \end{equation*}\end{linenomath}
        the \emph{generalized Wu functions}.
        \end{rem}
    
    As Eq. \eqref{defn:wu2} represents $\Phi(x)=\varphi_{\ell,k}(\norm{x}{2})$ as the self convolution of $P(x)=\sqrt{c_{\ell,k}}f_{\ell-k}(\norm{x}{2})$, Theorem \ref{thm:native_space_wu} could be generalized to all nonnegative integer dimensions.
    \begin{thm}\label{thm:native_space_gen_wu_func}
        Let $\varphi_{\ell,k}$ be a generalized Wu function with $k\in\mathbb{N}/2,\ \ell\in\R_+$ and $k\leq\ell$. Then the native space of $\Phi(x)=\varphi_{\ell,k}(\norm{x}{2})$ on $\R^{2k+1}$ is
        \begin{equation}
            \mathcal{N}_\Phi(\R^{2k+1})=P*L^2(\R^{2k+1}):=\{P*g:g\in L^2(\R^{2k+1})\},
            \end{equation}
        where 
        \begin{linenomath}\begin{equation*}
            \begin{gathered}
                P(x)=\frac{\Gamma(\ell+1)}{\Gamma(\ell-k+1)}(2/\pi)^{k/2}f_{\ell-k}(\norm{x}{2}),\\
                f_{\ell-k}(r)=(1-r^2)_+^{\ell-k}.
                \end{gathered}
        \end{equation*}\end{linenomath}
        Moreover, the convolution operator
        \begin{linenomath}\begin{equation*}
            C_P:L^2(\R^{2k+1})\to\mathcal{N}_\Phi(\R^{2k+1}),\quad C_Pg=P*g
            \end{equation*}\end{linenomath}
        is a bijective isometry.
        \end{thm}
    The proof of this theorem is parallel to the proof of Theorem \ref{thm:native_space_wu} and we shall omit the proof here.

%\begin{comment}
% section 3
\subsection{Closed form of generalized Wu function}

    In this part we will present the closed form of generalized Wu functions. Denote the \emph{Pochhammer symbol} by
    \begin{linenomath}\begin{equation*}
        a^{(\ell)}=\begin{cases}
            1,&\ell=0,\\
            a(a+1)\dots(a+\ell-1),&\ell\geq 1.
            \end{cases}
        \end{equation*}\end{linenomath}
    Obviously if $\Gamma(a)$ is well-defined, then
    \begin{linenomath}\begin{equation*}
        a^{(\ell)}=\frac{\Gamma(a+\ell)}{\Gamma(a)}.
        \end{equation*}\end{linenomath}
    With the Pochhammer symbol, the \emph{hypergeometric function} ${_2F_1}(a,b;c;z)$ could be expanded into power series \cite{BW35GE}
    \begin{linenomath}\begin{equation*}
        {_2F_1}(a,b;c;z)=\sum_{n=0}^\infty\frac{a^{(n)}b^{(n)}}{c^{(n)}}\cdot\frac{1}{n!}z^n.
        \end{equation*}\end{linenomath}
    Note that if there is a nonpositive integer among $a$ and $b$, then the function ${_2F_1}(a,b;c;z)$ is a polynomial since the Pochhammer symbol $a^{(n)}$ vanishes for $n\geq a+1$, and thus the power expansion of this function has finitely many nonzero terms.
    \begin{thm}\label{thm:phill}
        For integer $\ell\geq 1$, the Wu function $\varphi_{\ell,\ell}$ has a closed form
        \begin{equation}\label{eq:phill}
            \varphi_{\ell,\ell}(r)=2^{\ell+1}\Gamma(\ell+1)\left[\frac{1^{(\ell)}}{(3/2)^{(\ell)}}-\frac{r}{2}\cdot{_2F_1}\left(-\ell,\frac{1}{2};\frac{3}{2};\frac{r^2}{4}\right)\right],\quad r\leq 2.
            \end{equation}
        \end{thm}
    \begin{proof}
        For this case, we use the convolutional representation of Wu functions \eqref{defn:wu2}
        \begin{linenomath}\begin{equation*}
            \varphi_{\ell,\ell}=\Gamma(\ell+1)^2(2/\pi)^\ell\mathscr{C}_{2\ell+1}(f_0,f_0).
            \end{equation*}\end{linenomath}
        Notice that $f_0$ is the indicator function of one-dimensional unit ball, so the function $\mathscr{C}_{2\ell+1}(f_0,f_0)$ is the self convolution of the indicator function of $2\ell+1$-dimensional unit ball. Denote 
        \begin{linenomath}\begin{equation*}
            V_d(r)=\frac{\pi^{d/2}}{\Gamma(d/2+1)}r^d
            \end{equation*}\end{linenomath}
        as the volume of a $d$-dimensional ball with radius $r$. Suppose $\alpha\in\R^{2\ell+1}$ is a unit vector. Applying Fubini's theorem leads to
        \begin{linenomath}\begin{equation*}
            \begin{aligned}
                \mathscr{C}_{2\ell+1}(f_0,f_0)(r)&=\int_{\R^{2\ell+1}}f_0(\norm{y}{2})f_0(\norm{r\alpha-y}{2})\ dy\\
                &=2\int_{r/2}^1V_{2\ell}(\sqrt{1-x^2})\ dx\\
                &=\frac{2\pi^{\ell}}{\Gamma(\ell+1)}\int_{r/2}^1(1-x^2)^\ell\ dx,\quad r\leq 2.
                \end{aligned}
            \end{equation*}\end{linenomath}
        Notice that (see Figure \ref{mma-1} in Appendix for a simple Mathematica code)
        \begin{linenomath}\begin{equation*}
            \begin{aligned}
                \frac{d}{dx}\left[x\cdot{_2F_1}\left(-n,\frac{1}{2};\frac{3}{2};x^2\right)\right]&=\frac{d}{dx}\left(\sum_{j=0}^n\frac{(-n)^{(j)}(1/2)^{(j)}}{(3/2)^{(j)}}\frac{1}{j!}x^{2j+1}\right)\\
                &=\frac{\Gamma(3/2)}{\Gamma(1/2)}\sum_{j=0}^n\frac{(-n)^{(j)}}{j+1/2}\frac{2j+1}{j!}x^{2j}\\
                &=\sum_{j=0}^n\binom{n}{j}(-x^2)^j=(1-x^2)^n,
                \end{aligned}
            \end{equation*}\end{linenomath}
        and Gauss's summation theorem \cite[Thm 1.3]{BW35GE} implies
        \begin{linenomath}\begin{equation*}
            {_2F_1}\left(-\ell,\frac{1}{2};\frac{3}{2};1\right)=\frac{\Gamma(3/2)\Gamma(\ell+1)}{\Gamma(1)\Gamma(\ell+3/2)}=\frac{1^{(\ell)}}{(3/2)^{(\ell)}},
            \end{equation*}\end{linenomath}
        so
        \begin{linenomath}\begin{equation*}
            \mathscr{C}_{2\ell+1}(f_0,f_0)(r)=\frac{2\pi^\ell}{\Gamma(\ell+1)}\left[\frac{1^{(\ell)}}{(3/2)^{(\ell)}}-\frac{r}{2}\cdot{_2F_1}\left(-\ell,\frac{1}{2};\frac{3}{2};\frac{r^2}{4}\right)\right],
            \end{equation*}\end{linenomath}
        and the Wu function $\varphi_{\ell,\ell}$ follows.
        \end{proof}
    \begin{thm}\label{thm:closed_form_wu_func}
        For $\ell\in\mathbb{N}$ and $k\in\R_+$ such that $k\leq\ell$, the generalized Wu function $\varphi_{\ell,k}$ has a closed form
        \begin{linenomath}
        \begin{multline*}
            \varphi_{\ell,k}(r)=\frac{2^{2\ell-k+1}\Gamma(\ell+1)}{\Gamma(\ell-k+1)}\left(1-\frac{r^2}{4}\right)^{\ell-k}\left[\frac{1^{(\ell)}}{(3/2)^{(\ell)}}\right.\\
            \left.-\sum_{n=0}^\ell\frac{(-\ell)^{(n)}(1/2)^{(n)}}{(3/2)^{(n)}}\frac{1}{n!}\cdot{_2F_1}\left(-n-\frac{1}{2},1;\ell-k+1;1-\frac{r^2}{4}\right)\right],\quad r\leq 2.
            \end{multline*}
        \end{linenomath}
        \end{thm}
    \begin{proof}
        For this case, we use the differential representation of Wu functions \eqref{defn:wu1}
        \begin{linenomath}\begin{equation*}
            \varphi_{\ell,k}=\mathscr{D}^k(f_\ell*f_\ell).
            \end{equation*}\end{linenomath}
        The $f$-form of $\varphi_{\ell,k}$ is
        \begin{linenomath}\begin{equation*}
            \fform\varphi_{\ell,k}=I_{-k}\fform(f_\ell*f_\ell)=I_{\ell-k}I_{-\ell}\fform(f_\ell*f_\ell)=I_{\ell-k}\fform\varphi_{\ell,\ell},
            \end{equation*}\end{linenomath}
        so we just need to apply $I_{\ell-k}$ on $\fform\varphi_{\ell,\ell}$ to get the desired function. The $f$-form of $\varphi_{\ell,\ell}$ is
        \begin{linenomath}\begin{equation*}
            \begin{aligned}
                \fform\varphi_{\ell,\ell}(r)&=2^{\ell+1}\Gamma(\ell+1)\left[\frac{1^{(\ell)}}{(3/2)^{(\ell)}}-\sqrt{\frac{r}{2}}\cdot{_2F_1}\left(-\ell,\frac{1}{2};\frac{3}{2};\frac{r}{2}\right)\right]\\
                &=2^{\ell+1}\Gamma(\ell+1)\left[\frac{1^{(\ell)}}{(3/2)^{(\ell)}}-\sum_{n=0}^\ell\frac{(-\ell)^{(n)}(1/2)^{(n)}}{(3/2)^{(n)}}\frac{1}{n!}\left(\frac{r}{2}\right)^{n+1/2}\right].
                \end{aligned}
            \end{equation*}\end{linenomath}
        According to Euler's integral for hypergeometric functions \cite[Thm 1.5]{BW35GE}, for $\alpha\geq 0$, the action of $I_\alpha$ on powers of $r/2$ are (see Figure \ref{mma-1} in Appendix for a simple Mathematica code)
        \begin{linenomath}\begin{equation*}
            \begin{aligned}
                \int_r^2\frac{(x-r)^{\alpha-1}}{\Gamma(\alpha)}\left(\frac{x}{2}\right)^n\ dx&=\frac{1}{\Gamma(\alpha)}\int_0^{2-r}x^{\alpha-1}\left(\frac{x+r}{2}\right)^n\ dx\\
                &=\frac{(2-r)^\alpha}{\Gamma(\alpha)}\int_0^1x^{\alpha-1}\left[\frac{(2-r)x+r}{2}\right]^n\ dx\\
                &=\frac{(2-r)^\alpha}{\Gamma(\alpha)}\int_0^1(1-x)^{\alpha-1}\left[1-\left(1-\frac{r}{2}\right)x\right]^n\ dx\\
                &=\frac{(2-r)^\alpha}{\Gamma(\alpha+1)}\cdot{_2F_1}\left(-n,1;\alpha+1;1-\frac{r}{2}\right),\quad r\leq 2,
                \end{aligned}
            \end{equation*}\end{linenomath}
        so the action of $I_{\ell-k}$ on $\fform\varphi_{\ell,\ell}$ is
        \begin{linenomath}
        \begin{multline*}
            I_{\ell-k}\fform\varphi_{\ell,\ell}(r)=\frac{2^{2\ell-k+1}\Gamma(\ell+1)}{\Gamma(\ell-k+1)}\left(1-\frac{r}{2}\right)^{\ell-k}\left[\frac{1^{(\ell)}}{(3/2)^{(\ell)}}\right.\\
            \left.-\sum_{n=0}^\ell\frac{(-\ell)^{(n)}(1/2)^{(n)}}{(3/2)^{(n)}}\frac{1}{n!}\cdot{_2F_1}\left(-n-\frac{1}{2},1;\ell-k+1;1-\frac{r}{2}\right)\right],\quad r\leq 2,
            \end{multline*}
        \end{linenomath}
        and $\varphi_{\ell,k}$ follows by substituting $r$ with $r^2/2$.
        \end{proof}
    
    Some Wu functions and missing Wu functions are listed in Table \ref{table:wu_func}. The functions in Table \ref{table:wu_func} are rescaled to be supported on the unit interval. These functions are computed with Mathematica code in Figure \ref{mma-2}.

    For nonnegative half integers $k$ and nonnegative integers $\ell$, the generalized Wu functions $\varphi_{\ell,k}$, which are also called the missing Wu functions, share a common representation
    \begin{equation}\label{eq:closed_form_missing_wu_func}
        \varphi_{\ell,k}(2r)=p_{\ell,k}(2r)L(2r)+q_{\ell,k}(2r)S(2r)
        \end{equation}
    with polynomials $p_{\ell,k}(2r), q_{\ell,k}(2r)$ and 
    \begin{linenomath}\begin{equation*}
        \begin{gathered}
            L(2r)=\log\left(\frac{r}{1+\sqrt{1-r^2}}\right)=-\tanh^{-1}\left(\sqrt{1-r^2}\right),\\
            S(2r)=\sqrt{1-r^2}.
            \end{gathered}
        \end{equation*}\end{linenomath}
    Note that here we rescale the support of functions to the unit interval $[0,1]$. Here we do not give a rigorous proof for Eq. \eqref{eq:closed_form_missing_wu_func}, but it could be proved by induction that the functions ${_2F_1}(-n-1/2,1;\ell-k+1;1-r/2)$ has the same representation as Eq. \eqref{eq:closed_form_missing_wu_func}.
    
    The rescaled missing Wu functions shares a similar representation with the missing Wendland functions. In fact, we have the following expression for (missing) Wu functions in terms of (generalized) Wendland functions $\psi_{\mu,\alpha}^{wd}$ \cite{SR11MI}.
    
\section{Connection between Wu and Wendland functions}
    \begin{thm}\label{thm:wu_wendland}
        Let $\varphi_{\ell,k}$ denote the generalized Wu functions and $\psi_{\mu,\alpha}^{wd}$ denote the generalized Wendland functions. Then the equality
        \begin{linenomath}\begin{equation*}
            \varphi_{\ell,k}(2r)=2^{2\ell-k+1}\Gamma(\ell+1)\sum_{n=0}^\ell\binom{\ell}{n}\frac{1}{\ell+n+1}2^{\ell-n}(-1)^n\psi_{\ell+n+1,\ell-k}^{wd}(r)
            \end{equation*}\end{linenomath}
        holds for $k\in\mathbb{N}/2,\ \ell\in\mathbb{N}$ such that $\ell\geq k$. 
        \end{thm}
    \begin{proof}
        For $0\leq r\leq 1$, the representation \eqref{defn:wu2} implies
        \begin{linenomath}\begin{equation*}
            \begin{aligned}
                \varphi_{\ell,\ell}(2r)&=2^{\ell+1}\Gamma(\ell+1)\int_r^1(1-x)^\ell(1+x)^\ell\ dx\\
                &=2^{\ell+1}\Gamma(\ell+1)\cdot(-1)^\ell\sum_{n=0}^\ell\binom{\ell}{n}(-2)^{\ell-n}\int_r^1(1-x)^{\ell+n}\ dx\\
                &=2^{\ell+1}\Gamma(\ell+1)\sum_{n=0}^\ell\binom{\ell}{n}\frac{1}{\ell+n+1}2^{\ell-n}(-1)^n(1-r)^{\ell+n+1}\\
                &=2^{\ell+1}\Gamma(\ell+1)\sum_{n=0}^\ell\binom{\ell}{n}\frac{1}{\ell+n+1}2^{\ell-n}(-1)^n\psi_{\ell+n+1,0}^{wd}(r).
                \end{aligned}
            \end{equation*}\end{linenomath}
        Notice that for nonnegative integers $k\leq\ell$,
        \begin{linenomath}\begin{equation*}
            I_{\ell-k}\fform[\varphi_{\ell,\ell}(2t)](r)=I_{\ell-k}[\fform\varphi_{\ell,\ell}(2t)](r)=2^{k-\ell}I_{\ell-k}\fform\varphi_{\ell,\ell}(2r)=2^{k-\ell}\fform\varphi_{\ell,k}(2r),
            \end{equation*}\end{linenomath}
        so
        \begin{linenomath}\begin{equation*}
            \begin{aligned}
                \varphi_{\ell,k}(2r)&=2^{\ell-k}\fform^{-1}I_{\ell-k}\fform[\varphi_{\ell,\ell}(2t)](r)\\
                &=2^{2\ell-k+1}\Gamma(\ell+1)\sum_{n=0}^\ell\binom{\ell}{n}\frac{1}{\ell+n+1}2^{\ell-n}(-1)^n\fform^{-1}I_{\ell-k}\fform\psi_{\ell+n+1,0}^{wd}(r)\\
                &=2^{2\ell-k+1}\Gamma(\ell+1)\sum_{n=0}^\ell\binom{\ell}{n}\frac{1}{\ell+n+1}2^{\ell-n}(-1)^n\psi_{\ell+n+1,\ell-k}^{wd}(r),
                \end{aligned}
            \end{equation*}\end{linenomath}
        according to the definition of the generalized Wendland function \cite{SR11MI}.
        \end{proof}

    \section{Discussion}

    This paper aims to derive an explicit characterization of the native spaces of Wu functions and to generalize the Wu functions such that ``the missing Wu functions" in even-dimensional spaces can be found. To this end, we proposed several questions in the introduction, each concerning a specific problem related to Wu functions. Since the RBF theory provides a unique perspective to understand the learning algorithms \cite{PJ93AP,HC93AP,ZD07LE}, especially the multilayer networks \cite{PT90RE}, answering these questions not only provide a mathematical rigor, but may also provide some potential tools or insights to approximation and learning theory. 
    
    The Bochner's characterization based on Fourier transform and the reproducing kernel Hilbert space theory on equivalent kernels serve as the basic tools to study native spaces: whether one native space includes another depends on whether its kernel is `greater' than that of another space. The difficulty lies in how to identify the Fourier transform of Wu functions; and how to compare the Fourier transform of a Wu function and that of the kernel of a Sobolev space --- an inverse multiquadric. By comparing the kernels, we conclude that the native space of a Wu function is a dense subspace of a Sobolev space. Such a comparison largely depends on the property of some special functions and their Fourier transforms.

    The $f$-form operator significantly simplifies many operations on radial functions and their Fourier transforms. Moreover, the $f$-form operator clarifies the RBF construction strategy and results in the more accessible $D:=I_{-1}$ and $I_{1}$ operators in Proposition 2.2. They enjoy better dimensional mobility: they can ``walk" between all integer-dimensional spaces, rather than only odd-dimensional spaces. Such a dimensional mobility between all spaces makes it possible to find the missing Wu functions and all the generalized Wu functions. In fact both Wu's and Wendland's constructions were partially published, since the ``missing functions" in even-dimensional spaces involve non-polynomial terms, they were not so welcomed as polynomials and thus not published instantly after being noticed.
    
    The closed form of the generalized Wu functions largely depends on some sophisticated mathematics involving the hypergeometric functions. Such fundamental and classical function theory should have received more attention. The relationship between Wu functions and Wendland functions are revealed by linking and comparing their closed forms. Finally we want to mention that the theory on Wu functions is not yet completed. One of the interesting questions is whether it is possible construct CSRBFs in terms of polynomials in the even dimensional spaces. 
    %An interesting problem is on the theory of Wu functions for the sphere. One may want to refer to Hubbert's work for more background \cite{HS03GE}.

%%%%%%%%%%%%%%%%%%%%%%%%%%%%%%%%%%%%%%%%%%%%%%%%%%%%%%%
%%% Acknowledgements.
%%%%%%%%%%%%%%%%%%%%%%%%%%%%%%%%%%%%%%%%%%%%%%%%%%%%%%%
\Acknowledgements{This work is funded by the natural science foundation of China (12271047); Guangdong Provincial Key Laboratory of Interdisciplinary Research and Application for Data Science, BNU-HKBU United International College (2022B1212010006); UIC research grant (R0400001-22; UICR0400008-21; R72021114); Guangdong College Enhancement and Innovation Program (2021ZDZX1046).}

%%%%%%%%%%%%%%%%%%%%%%%%%%%%%%%%%%%%%%%%%%%%%%%%%%%%%%%
%%% Conflict of interest.
%%%%%%%%%%%%%%%%%%%%%%%%%%%%%%%%%%%%%%%%%%%%%%%%%%%%%%%
%\InterestConflict

%%%%%%%%%%%%%%%%%%%%%%%%%%%%%%%%%%%%%%%%%%%%%%%%%%%%%%%
%%% Supplements.
%%%%%%%%%%%%%%%%%%%%%%%%%%%%%%%%%%%%%%%%%%%%%%%%%%%%%%%
%\Supplements{}

%%%%%%%%%%%%%%%%%%%%%%%%%%%%%%%%%%%%%%%%%%%%%%%%%%%%%%%
%%% Reference section.
%%% citation in the content using "some words~\cite{1,2}".
%%% ~ is needed to make the reference number is on the same line with the word before it.
%%%%%%%%%%%%%%%%%%%%%%%%%%%%%%%%%%%%%%%%%%%%%%%%%%%%%%%

\bibliographystyle{abbrv}
\bibliography{refs}

%%%%%%%%%%%%%%%%%%%%%%%%%%%%%%%%%%%%%%%%%%%%%%%%%%%%%%%
%%% Appendix sections.
%%%%%%%%%%%%%%%%%%%%%%%%%%%%%%%%%%%%%%%%%%%%%%%%%%%%%%%
\begin{appendix}

\section{}

    \begin{figure}[H]
        \centering
        \vspace{-0.5cm}
        \begin{minipage}{1\linewidth}
            \centering
            \includegraphics[width=1\textwidth]{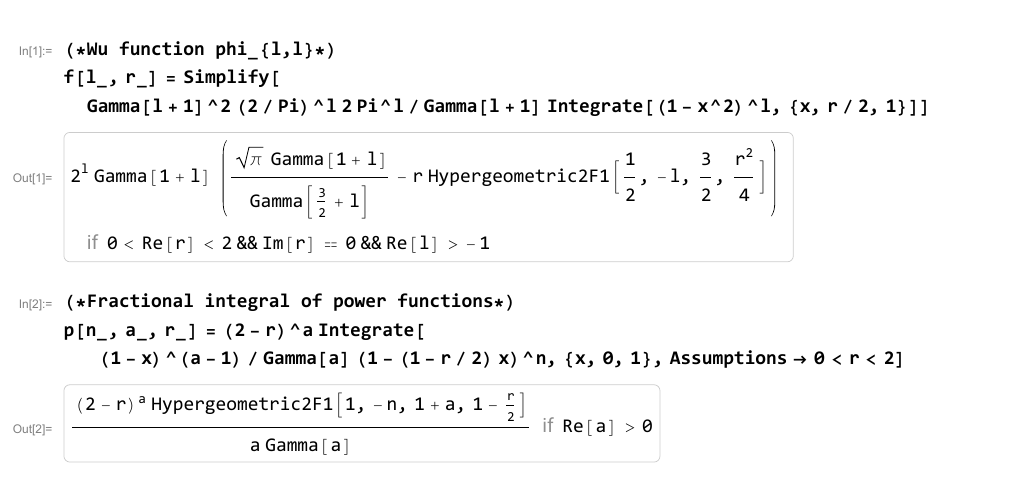}
            \caption{Mathematica code for computing $\varphi_{\ell,\ell}$ and the fractional integral of power functions}
            \label{mma-1}
            \end{minipage}
        \begin{minipage}{1\linewidth}
            \centering
            \includegraphics[width=1\textwidth]{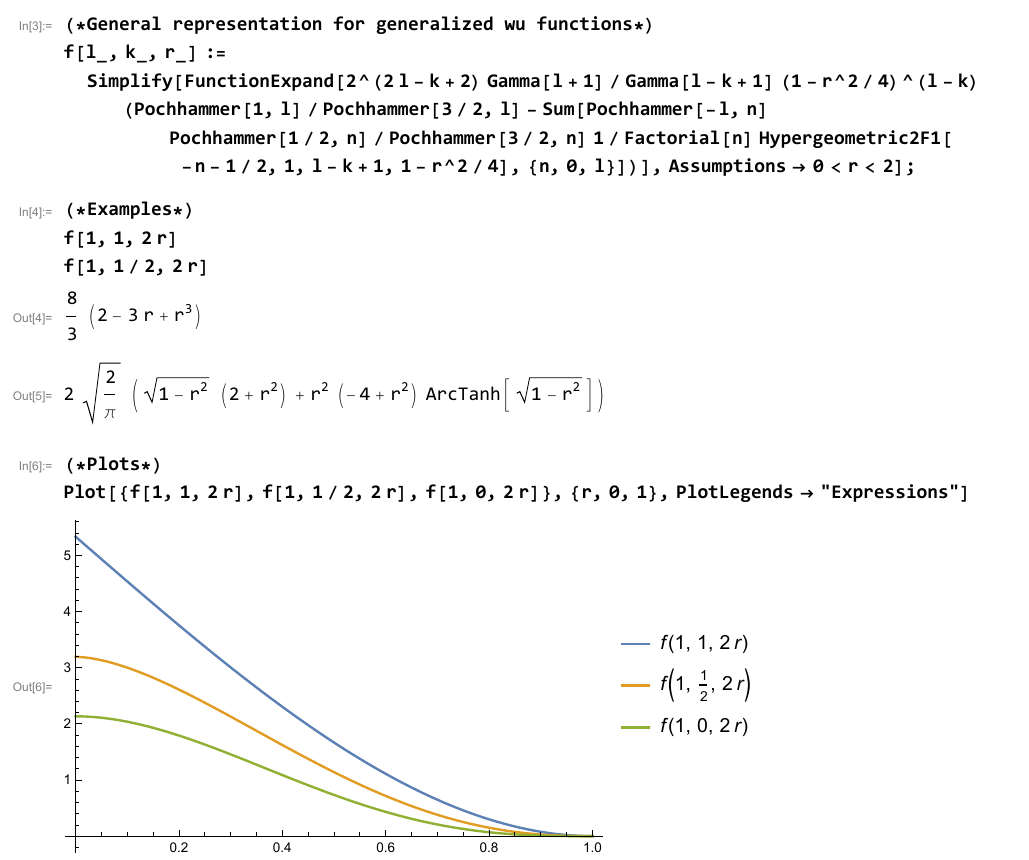}
            \caption{Mathematica code for computing generalized Wu functions}
            \label{mma-2}
            \end{minipage}
            \vspace{-1cm}
        \end{figure}

    \begin{table}[H]
        \caption{Some small positive zeros of $J_\nu$.}
        \label{table:zeros_jnu}
        \centering
        \begin{tabular}{c|cccccc}
            \toprule
            $j_{\nu,k}$ & $k=1$ & $k=2$ & $k=3$ & $k=4$ & $k=5$ & $k=6$ \\
            \midrule
            $\nu=0.0$ & 2.40483 & 5.52008 & 8.65373 & 11.7915 & 14.9309 & 18.0711 \\
            $\nu=0.5$ & 3.14159 & 6.28319 & 9.42478 & 12.5664 & 15.708 & 18.8496 \\
            $\nu=1.0$ & 3.83171 & 7.01559 & 10.1735 & 13.3237 & 16.4706 & 19.6159 \\
            $\nu=1.5$ & 4.49341 & 7.72525 & 10.9041 & 14.0662 & 17.2208 & 20.3713 \\
            $\nu=2.0$ & 5.13562 & 8.41724 & 11.6198 & 14.796 & 17.9598 & 21.117 \\
            $\nu=2.5$ & 5.76346 & 9.09501 & 12.3229 & 15.5146 & 18.689 & 21.8539 \\
            \midrule
            \multicolumn{7}{l}{*$j_{\nu,k}$ denotes the $k$-th positive zero of $J_\nu$.}\\
            \bottomrule
            \end{tabular}
        \end{table}

    \linespread{1.5}
    \begin{table}[H]
        \centering
        \caption{Rescaled Wu functions and missing Wu functions.}
        \label{table:wu_func}
        \begin{tabular}{c|l}
            \toprule
            $d$ & $\varphi_{\ell,k}(2r)=\varphi_{\ell,(d-1)/2}(2r),r\leq 1$. $\varphi_{\ell,k}(r)$ denotes the generalized Wu function.\\
            \midrule
            \multirow{3}{*}{$d=1$}  & $\varphi_{0,0}(2r)\quad\!=4(1-r)$                                        \\
                                    & $\varphi_{1,0}(2r)\quad\!=\frac{32}{15}(1-r)^3(1+3r+r^2)$                \\
                                    & $\varphi_{2,0}(2r)\quad\!=\frac{512}{315}(1-r)^5(1+5r+9r^2+5r^3+r^4)$    \\
            \midrule
            \multirow{5}{*}{$d=2$}  & $\varphi_{1,1/2}(2r)=2\sqrt{\frac{2}{\pi}}[(2+r^2)S(2r)+r^2(4-r^2)L(2r)]$                                    \\
                                    & $\varphi_{2,1/2}(2r)=\frac{2}{9}\sqrt{\frac{2}{\pi}}[(16-88r^2-42r^4+9r^6)S(2r)$\\
                                    & \multicolumn{1}{r}{}$+3r^4(-48+16r^2-3r^4)L(2r)]$   \\
                                    & $\varphi_{3,1/2}(2r)=\frac{2}{75}\sqrt{\frac{2}{\pi}}[(128-896+3168r^4+1480r^6-490r^8+75r^{10})S(2r)$        \\
                                    & \multicolumn{1}{r}{}$+15r^6(320-120r^2+36r^4-5r^6)L(2r)]$                                                 \\
            \midrule
            \multirow{3}{*}{$d=3$}  & $\varphi_{1,1}(2r)\quad\!=\frac{8}{3}(1-r)^2(2+r)$                                   \\
                                    & $\varphi_{2,1}(2r)\quad\!=\frac{128}{105}(1-r)^4(4+16r+12r^2+3r^3)$                  \\
                                    & $\varphi_{3,1}(2r)\quad\!=\frac{1024}{1155}(1-r)^6(6+36r+82r^2+72r^3+30r^4+5r^5)$    \\
            \midrule
            \multirow{6}{*}{$d=4$}  & $\varphi_{2,3/2}(2r)=\frac{4}{3}\sqrt{\frac{2}{\pi}}[(8+10r^2-3r^4)S(2r)+3r^2(8-4r^2+r^4)L(2r)]$ \\
                                    & $\varphi_{3,3/2}(2r)=\frac{2}{5}\sqrt{\frac{2}{\pi}}[(32-224r^2-188r^4+80r^6-15r^8)S(2r)$        \\
                                    & \multicolumn{1}{r}{} $+15r^4(-32+16r^2-6r^4+r^6)L(2r)]$                                       \\
                                    & $\varphi_{4,3/2}(2r)=\frac{8}{525}\sqrt{\frac{2}{\pi}}[q_{4,3/2}(2r)S(2r)+p_{4,3/2}(2r)L(2r)]$   \\
                                    & $q_{4,3/2}(2r)\,=1024-8448r^2+36224r^4+25520r^6-12600r^8+3850r^{10}-525r^{12}$                  \\
                                    & $p_{4,3/2}(2r)\,=105r^6(640-320r^2+144r^4-40r^6+5r^8)$                                          \\
            \midrule
            \multicolumn{2}{l}{*$S(2r)=\sqrt{1-r^2},\ L(2r)=\log\left(\frac{r}{1+\sqrt{1-r^2}}\right)$}\\
            \bottomrule
            \end{tabular}
        \end{table}
    
\end{appendix}

\end{document}